\documentclass[12pt]{amsart}
\usepackage{amsmath,amsthm}
\usepackage{amsbsy}
\usepackage{amsfonts}
\usepackage{amstext}
\usepackage{amssymb}
\usepackage{MnSymbol}
\usepackage{mathbbol}
\usepackage{wasysym}
\usepackage{eucal}
\usepackage{mathtools}
\usepackage{stackrel}
\usepackage{ulem}

\usepackage{tikz, tikz-cd}
   \usetikzlibrary{matrix}
   \usetikzlibrary{calc}
   \usetikzlibrary{decorations}

\usepackage{float}
\usepackage{subfigure}
\usepackage{pdflscape}
\usepackage{enumitem}

\usepackage{placeins} 

\allowdisplaybreaks

\usepackage[margin=1in]{geometry} 

\usepackage{color}
\usepackage{multirow,bigdelim}

\numberwithin{equation}{section} 
\newtheorem{thm}[equation]{Theorem} 
\newtheorem*{theorem*}{Theorem}

\newtheorem{prop}[equation]{Proposition}
\newtheorem{lemma}[equation]{Lemma}

\newtheorem{definition}[equation]{Definition}

\newtheorem{notation}[equation]{Notation}  

\theoremstyle{definition}
\newtheorem{method}[equation]{Method} 
\newtheorem{rem}[equation]{Remark}

\renewcommand{\hom}[1]{\ensuremath{\textnormal{Hom}_{#1}}}
\newcommand{\ext}[1]{\ensuremath{\textnormal{Ext}_{#1}}}

\renewcommand{\ker}{\ensuremath{\textnormal{ker}}}
\newcommand{\im}{\ensuremath{\textnormal{im}}}

\newcommand{\ol}[1]	{\ensuremath{\overline{#1}}}

\newcommand{\p}{\ensuremath{^{\prime}}}

\newcommand{\opp}{\ensuremath{{}^\text{op}}}

\newcommand{\rar}{\rightarrow}

\newcommand{\srar}{\twoheadrightarrow}

\newcommand{\brac}[1]{\ensuremath{\langle #1 \rangle}}

\newcommand{\ground}{\ensuremath{\mathbb k}}		
\newcommand{\quiv}	{\ensuremath{Q}} 				
\newcommand{\alg}{\ensuremath{A}}			
\newcommand{\coho}{\ensuremath{\mathsf{H}}}		
\newcommand{\hh}{\ensuremath{\mathsf{HH}}}		
\newcommand{\barcx}{\ensuremath{\mathbb{B}}}		
\newcommand{\bard}{\ensuremath{d}}			
\newcommand{\bas}[1]{\ensuremath{{\mathcal {B}}_{#1}}}	
\newcommand{\g}[2]{\ensuremath{g_{#1}^{#2}}}		
\newcommand{\vir}{\textnormal{Vir}}

\newcommand{\ZZ}{\ensuremath{{\mathbb Z}}}
\newcommand{\id}{\mathrm{id}}
\newcommand{\T}{\otimes}

\begin{document}



\title[Gerstenhaber structure of a class of special biserial algebras]{The Gerstenhaber structure on the Hochschild cohomology of a class of special biserial algebras}

\author[J.\ Meinel]{Joanna Meinel}
\address{Mathematical Institute, University of Bonn, Germany}
\email{joanna@math.uni-bonn.de}

\author[V.\ C.\ Nguyen]{Van C.~Nguyen}
\address{Department of Mathematics, Hood College, Frederick, MD 21701, U.S.A.}
\email{nguyen@hood.edu}

\author[B. Pauwels]{Bregje Pauwels} 
\address{Mathematical Sciences Institute, The Australian National University, Acton ACT 2601, Australia}
\email{bregje.pauwels@anu.edu.au}

\author[M.\ J.\ Redondo]{Maria Julia Redondo}
\address{Departamento de Matem\'atica, Instituto de Matem\'{a}tica (INMABB), Universidad Nacional del Sur (UNS)-CONICET, Bah\'{i}a Blanca, Argentina}
\email{mredondo@uns.edu.ar}

\author[A.\ Solotar]{Andrea Solotar}
\address{Departamento de Matem\'{a}tica, Universidad de Buenos Aires and IMAS, CONICET/UBA, Buenos Aires, Argentina}
\email{asolotar@dm.uba.ar}

\date{\today}


\begin{abstract}
We determine the Gerstenhaber structure on the Hochschild cohomology ring of a class of self-injective special biserial algebras. Each of these algebras is presented as a quotient of the path algebra of a certain quiver. In degree one, we show that the cohomology is isomorphic, as a Lie algebra, to a direct sum of copies of a subquotient of the Virasoro algebra. These copies share Virasoro degree 0 and commute otherwise. Finally, we describe the cohomology in degree $n$ as a module over this Lie algebra by providing its decomposition as a direct sum of indecomposable modules.
\end{abstract}

\maketitle



\section{Introduction}
\label{sec:intro}

Let $\ground$ be a field and fix an integer $m\geq 1$. Following \cite{ST}, we consider the path algebra 
$\ground\quiv$ of the quiver $\quiv$ with $m$ vertices $\{0,1,\ldots, m-1\}$ and $2m$ arrows given by $(a_i: i \rightarrow i+1)$ and $(\ol{a}_i: i+1 \rightarrow i)$ for $i\in\{0,1,\ldots,m-1\}$. The composition of arrows is ordered from left to right, for example $a_ia_{i+1}$ denotes the path $i \xrightarrow{a_i} i+1 \xrightarrow{a_{i+1}} i+2$. For the rest of the paper, we take all indices modulo $m$, that is, we identify the labeling set of the vertices with $\ZZ/m\ZZ$.

\begin{figure}[H]\label{fig:quiver}
\begin{center}
\begin{tikzpicture}[scale=0.6]
\begin{scope}[xshift=0cm]
\path[use as bounding box] (-3.5,-1.5) rectangle (3.5,2.5);
\foreach \x in {6,7,0,1,2}
\fill ({90-\x*45}:2.5cm) circle (1.2mm);

\foreach \x in {5,6,7,0,1,2}
\draw[thick, ->] ({85-\x*45}:2.65cm) arc ({85-\x*45}:{50-\x*45}:2.65cm);

\foreach \x in {5,6,7,0,1,2}
\draw[thick, <-] ({85-\x*45}:2.35cm) arc ({85-\x*45}:{50-\x*45}:2.35cm);

\draw[thick, dotted] ({85-3*45}:2.5cm) arc ({85-3*45}:{50-4*45}:2.5cm);

\draw ({90-1*45}:3cm) node {$i$};
\draw ({90-2*45}:3.35cm) node {$i+1$};

\draw (22:1.9cm) node {$\ol{a}_i$};
\draw (22:3.1cm) node {$a_i$};

\end{scope}

\end{tikzpicture}
\end{center} \vskip 5pt
\caption{The quiver $\quiv$}
\end{figure}

For any integer $N \geq1$, we consider the quotient algebra 
$$\alg_N:=\ground\quiv/ \brac {a_i a_{i+1},\ \ol{a}_{i+1}\ol{a}_i,\ (a_i\ol{a}_i)^N - (\ol{a}_{i-1}a_{i-1})^N \mid i \in\ZZ/m\ZZ}.$$
We note that the algebra $\alg_N$ is self-injective and special biserial. Special biserial algebras play an important role in the representation theory of algebras; for example,
the study of Hochschild cohomology for blocks of group algebras with cyclic or dihedral 
defect groups is achieved using the fact that the basic algebras that appear are special biserial \cite{Ho1, Ho2}. As a consequence, several articles have been devoted in recent years to the computation of the Hochschild cohomology of self-injective special biserial algebras, see for examples \cite{ES, X} and the references therein. 

The case $N=1$ has been studied in many different contexts, revealing interesting connections. In the special case where $m=1$, $\alg_1$ provides a counterexample to Happel's question, see \cite{BGSS}. When $m$ is even, $\alg_1$ appears in the presentation by quiver and relations of the Drinfeld double $\mathcal{D}(\Lambda_{h,d})$ of the Hopf algebra $\Lambda_{h,d}$, where  $d \mid h$, $h$ is even and $\Lambda_{h,d}$ is the algebra given by an oriented cycle with $h$ vertices such that all paths of length $d$ are zero, 
see~\cite{EGST}.  The quiver of $\mathcal{D}(\Lambda_{h,d})$ consists of $\frac{h^2}{d}$ isolated vertices and $\frac{h(d-1)}{2}$ copies of the quiver~$\quiv$ with $m=\frac{2h}{d}$ vertices, corresponding to the algebra $\alg_1$. 

For general $m$, the algebra $\alg_1$ also occurs in the study of the representation theory of $U_q(\mathfrak{sl}_2)$; see~\cite{ChK,P, Su}. Moreover, Gadbled, Thiel and Wagner describe three different gradings of $\alg_1$ in~\cite{GTW} -- we comment on these gradings in Section~\ref{brackets HH1} -- and
prove that the homotopy category $\mathcal{C}_m$ of finitely generated trigraded projective modules over $\alg_1$ carries an action of the extended affine braid group of Dynkin type $\widehat{\mathsf{A}}_m$.
The induced action of this braid group on the Grothendieck group of $\mathcal{C}_m$ is a $2$-parameter homological representation of the braid group of Dynkin type $\widehat{\mathsf{A}}_m$.
As mentioned in~\cite{GTW}, the algebra $\alg_1$ is a particular case of the general 
construction of algebras associated to graphs given in~\cite{HK}.

For general $m$ and $N$, the algebra $\alg_N$ occurs in work of Farnsteiner and Skowro\'nski, where they determine the Hopf algebras associated to infinitesimal groups whose principal blocks are tame when $\ground$ is an algebraically closed field with characteristic greater than or equal to~$3$, see~\cite{FS, FS2}.

Finally, Snashall and Taillefer prove in~\cite{ST} that the Hochschild cohomology $\hh^{\bullet}(\alg_N)$ of $\alg_N$ is finitely generated as an algebra over $\ground$. 
Their result allows the study of this graded commutative algebra with geometrical tools. In particular, they prove that the Hochschild cohomology algebra modulo nilpotents is, in this case, a finitely generated $\ground$-algebra and a commutative ring of Krull dimension $2$. \\
	
In this article, we provide a complete description of $\hh^{\bullet}(\alg_N)$ as a Gerstenhaber algebra. We assume that $m\geq 3$ and that the characteristic of $\ground$ does not divide $2$, $N$ or~$m$. The Gerstenhaber structure of an associative algebra was first considered in \cite{G}, and has been widely studied since then. 
However, the Gerstenhaber structure on the Hochschild cohomology of an algebra is very difficult to compute in concrete examples, since it is defined in terms of the bar resolution of the algebra. 
Recently, new computational methods have emerged, see for examples \cite{NW, S, V}. 
In particular, the method proposed by Su\'arez-\'Alvarez in \cite{S} allows us to compute
the Gerstenhaber brackets of two special elements $\varphi$ and $\psi$ in $\hh^{1}(\alg_N)$ with a $\ground$-basis of $\hh^{\bullet}(\alg_N)$.

In particular, we show that the derivations $[\varphi,-]$ and $[\psi,-]$ act diagonally on $\hh^{\bullet}(\alg_N)$ with respect to the basis given in~\cite{ST}. By the Jacobi identity, brackets of elements in the eigenbasis remain eigenvectors, and we can compute the corresponding eigenvalues.
This fact, together with the Poisson identity, allows us to compute the complete Gerstenhaber structure on $\hh^{\bullet}(\alg_N)$.
Moreover, we believe this method could be useful for other families of algebras with similar properties, providing an innovative approach to the computation of Gerstenhaber brackets on Hochschild cohomology.

Our computations provide a complete description of the Lie algebra structure of $\hh^{1}(\alg_N)$ and its Lie-action on $\hh^{\bullet}(\alg_N)$. Indeed, we find that:

\begin{theorem*} Let $N \geq 1$ and $m\geq 3$ be integers. Suppose the characteristic of a field $\ground$ does not divide $2$, $N$ or~$m$. Let $\alg_N$ be the quotient algebra defined over $\ground$ as above. Then the first Hochschild cohomology space $\hh^1(\alg_N)$ is embedded as a Lie algebra into a direct sum of a one-dimensional central Lie algebra $\langle c \rangle$ and $m$ copies of a subquotient $\mathfrak{a}_{N-1}$ of the Virasoro algebra, which share Virasoro degree 0 and commute otherwise, as follows:
\begin{align*}
\hh^1(\alg_N)\ &\hookrightarrow \ \langle c \rangle \oplus (\mathfrak{a}_{N-1})^{\oplus m}.
\end{align*}
\end{theorem*}
See Theorem~\ref{th:vir} for a precise statement. Note that when $N=1$, the Lie algebra $\hh^1(\alg_1)$ is commutative with two generators. It is only when $N>1$ that the Lie algebra becomes more complicated. It would be interesting to relate this Lie algebra structure to the representation theory of $\alg_N$, as has been done for blocks of group algebras in the modular case, see \cite{BKL}. \\

	
	The paper is structured as follows: we recall some basic facts about the Gerstenhaber structure on Hochschild cohomology in~Section~\ref{sec:Hochschild}, and describe Su\'arez-\'Alvarez's method from~\cite{S} in~Section~\ref{sec:approach}. 
	In~Sections~\ref{sec:resolution} and~\ref{cohomology}, we remind the readers of the description of the $\ground$-algebra $\hh^{\bullet}(\alg_N)$ given in~\cite{ST}.
	We compute the brackets of $\varphi$ and $\psi$ in $\hh^1(\alg_N)$ with the elements in a $\ground$-basis of $\hh^{\bullet}(\alg_N)$ in~Section~\ref{brackets HH1}.
	Section~\ref{general brackets} is devoted to the computation of all brackets among $\ground$-algebra generators of~$\hh^{\bullet}(\alg_N)$, providing the complete Gerstenhaber structure of $\hh^{\bullet}(\alg_N)$.
	In Section~\ref{E brackets}, we compute the Lie algebra structure of $\hh^{1}(\alg_N)$ and its Lie-action on $\hh^{n}(\alg_N)$ for all $n\geq 0$. 
	We gather all these results in Section~\ref{sec:lie-structure}, where we describe the Lie algebra $\hh^{1}(\alg_N)$ in terms of the Virasoro algebra, and provide a decomposition of $\hh^{n}(\alg_N)$ into indecomposable modules as a Lie module over~$\hh^{1}(\alg_N)$. \\

\textbf{Acknowledgement:} This project was started at the ``Women In Algebra and Representation Theory" workshop at Banff in March 2016. The authors thank Banff International Research Station, National Science Foundation, and the workshop organizers for this collaboration opportunity. The authors also thank Volodymyr Mazorchuk for helpful remarks.



\section{Hochschild cohomology}  
\label{sec:Hochschild}

Let $\alg$ be any associative algebra over $\ground$ and $\alg^e = \alg \otimes_\ground \alg\opp$ be its enveloping algebra, where $\alg\opp$ is $\alg$ with opposite multiplicative structure. The $n$-fold tensor product $A^{\otimes n}$ is a left $A^e$-module under left and right multiplication, equivalently, it is an $\alg$-$\alg$-bimodule. For $n\ge2$, it is free. There is an exact sequence of free left $\alg^e$-modules, called the \textbf{bar resolution} $\barcx(\alg)=(\alg^{\otimes(n+2)},\bard_n)_{n \geq 0}$, see for example~\cite{M}:
 
$$(\barcx(\alg) \rightarrow A): \qquad \cdots \xrightarrow{\bard_3} \alg^{\otimes 4} \xrightarrow{\bard_2} \alg^{\otimes 3} \xrightarrow{\bard_1} \alg \otimes \alg \xrightarrow{\bard_0} \alg \rightarrow 0,$$
where $\bard_0$ is multiplication, and the differential maps $\bard_n$ are given by:
$$\bard_n(a_0 \otimes a_1 \otimes \cdots \otimes a_{n+1}) = \sum_{i=0}^n (-1)^i a_0 \otimes a_1 \otimes \cdots \otimes a_ia_{i+1} \otimes \cdots \otimes a_{n+1},$$
for all $a_0, \ldots, a_{n+1} \in \alg$. Applying the functor $\hom{\alg^e}(-,\alg)$ to this bar resolution, one obtains a complex $\hom{\alg^e}(\barcx(\alg),\alg)$ with 
differentials $\bard_n^*(f) = f \circ \bard_n$, for any $f \in \hom{\alg^e}(\barcx_{n-1}(\alg),\alg)$, and $\bard_0^*$ is taken to be the zero map. The \textbf{$n$-th Hochschild cohomology} 
of $\alg$ is the $n$-th homology of this new complex 
$$\hh^n(\alg) := \coho_n(\hom{\alg^e}(\barcx(\alg),\alg)) = \ker(\bard_{n+1}^*)/ \im(\bard_n^*),$$
for all $n \geq 0$. It is well-known that the Hochschild cohomology $\hh^{\bullet}(\alg) := \bigoplus_{n \geq 0} \hh^n(\alg)$ is a graded commutative ring via the cup product, that is, $xy = (-1)^{|x||y|} yx$. Here, we denote $|x|=n$ to be the homological degree of element $x \in \hh^n(\alg)$. Moreover, in low degrees, it is well known that:
\begin{itemize}
 \item $\hh^0(\alg) = \ker(\bard_1^*) \cong Z(\alg)$ is the center of the algebra $\alg$; 
 \item $\hh^1(\alg)$ is the space of derivations of $\alg$ modulo its inner derivations;
 \item $\hh^2(\alg)$ is the space of equivalence classes of infinitesimal deformations of $\alg$, which plays an important role in the study of deformation theory of the algebra $\alg$ \cite{G2}.
\end{itemize} 

Besides the associative product structure, $\hh^{\bullet}(\alg)$ also has a bracket operation $[-,-]$, called the Gerstenhaber bracket, of degree $(-1)$. This bracket gives $\hh^{\bullet}(\alg)$ the structure of a graded Lie ring and makes $\hh^{\bullet}(\alg)$ into a Gerstenhaber algebra \cite{G}. The bracket operation is defined at the chain level on the bar resolution as follows. Let $f \in \hom{\alg^e}(\barcx_n(\alg), \alg)$ and $g \in  \hom{\alg^e}(\barcx_q(\alg), \alg)$ represent elements in $\hh^n(\alg)$ and $\hh^q(\alg)$ respectively. Their bracket $[f, g]$ is defined as an element of $\hom{\alg^e}(\barcx_{n+q-1}(\alg), \alg)$ given by
\[ [f, g] = f \circ  g - (-1)^{(n-1)(q-1)} g \circ f, \]
where the circle product $f \circ  g$ is defined by
\begin{align*}
 (f & \circ g) (1 \otimes a_1 \otimes \dots \otimes a_{n+q-1}\otimes 1) =  \\
 &\sum^n_{i=1}  (-1)^{(q-1)(i-1)} f(1 \otimes a_1 \otimes \dots \otimes  a_{i-1}\otimes  g(1 \otimes a_i \otimes \cdots \otimes a_{i+q-1} \otimes 1) \otimes a_{i+q} \otimes \cdots \otimes  a_{n+q-1} \otimes 1),
\end{align*}
and similarly for $g \circ f$. The Gerstenhaber bracket satisfies the following properties:
\begin{enumerate}
 \item Antisymmetry: $[x,y] = -(-1)^{(|x|-1)(|y|-1)} [y,x].$
 \item Poisson identity: $[xy,z] = [x,z]y + (-1)^{|x|(|z|-1)} x[y,z].$
 \item Jacobi identity: $$(-1)^{(|x|-1)(|z|-1)}[x,\, [y,z]] + (-1)^{(|y|-1)(|x|-1)}[y,\, [z,x]] + (-1)^{(|z|-1)(|y|-1)}[z,\, [x,y]] = 0.$$
\end{enumerate} 

Gerstenhaber brackets are in general difficult to compute due to the complexity of the bar resolution. One traditional approach is to construct explicit comparison maps to translate the brackets from the bar resolution to a more computationally friendly resolution. Such comparison maps are, in general, rather complicated to find. Recent progress has been made to define the Gerstenhaber brackets directly on \textit{any} projective resolution, see e.g.~\cite{NW, S, V}. 

We will investigate the Gerstenhaber structure of the Hochschild cohomology $\hh^{\bullet}(\alg)$ for our algebra $\alg:=\alg_N$ defined in~Section~\ref{sec:intro}. In order to compute the Gerstenhaber brackets $[\hh^1(\alg),-]$, we use a technique introduced in \cite{S} by Su\'{a}rez-\'{A}lvarez, which we discuss in the next section.



\section{The Gerstenhaber bracket: an approach by Su\'{a}rez-\'{A}lvarez}
\label{sec:approach}

In this section, we consider a $\ground$-algebra $B$ with a derivation $\delta: B\rar B$. If $W$ is a left $B$-module, a \textbf{$\delta$-operator} on $W$ is a $\ground$-linear map $f:W\rar W$ such that 
$$f(bw)=\delta(b)w+bf(w),$$
for all $b\in B$ and $w\in W$.
If $\epsilon:P_{\bullet}\srar W$ is a projective resolution of $W$,
$$\begin{tikzcd}
	\cdots \arrow{r} & P_2 \arrow{r}{d_2} & P_1 \arrow{r}{d_1} & P_0 \arrow{r}{\epsilon} & W \arrow{r} & 0,
\end{tikzcd}$$
then a \textbf{$\delta$-lifting} of $f$ to $P_{\bullet}$ is a sequence 
$f_{\bullet}=(f_n)_{n\geq 0}$ of $\delta$-operators $f_n:P_n\rar P_n$, such that the following diagram commutes:
$$\begin{tikzcd}
	\cdots \arrow{r} & P_2 \arrow{r}{d_2}\ \arrow{d}{f_2} & P_1 \arrow{r}{d_1} \arrow{d}{f_1} & P_0 \arrow{r}{\epsilon} \arrow{d}{f_0}  & W \arrow{r} \arrow{d}{f} & 0\\
	\cdots \arrow{r} & P_2 \arrow{r}{d_2} & P_1 \arrow{r}{d_1} & P_0 \arrow{r}{\epsilon} & W \arrow{r} & 0.
\end{tikzcd}$$
In~\cite{S}, Su\'{a}rez-\'{A}lvarez proves that $\delta$-liftings exist and are unique up to an equivalence:

\begin{lemma}\cite[Lemma 1.4]{S}
	Let $W$ be a left $B$-module and $\epsilon:P_{\bullet}\srar W$ be a projective resolution of $W$. If $f$ is a $\delta$-operator on~$W$, there exists a $\delta$-lifting $f_{\bullet}:P_{\bullet}\rar P_{\bullet}$ of $f$ to $P_{\bullet}$. Moreover, if $f_{\bullet}$ and $f'_{\bullet}$ are both $\delta$-liftings of $f$ to $P_{\bullet}$, then $f_{\bullet}$ and $f'_{\bullet}$ are $B$-linearly homotopic.
\end{lemma}

Now, suppose $f:W\rar W$ is a $\delta$-operator with $\delta$-lifting
$f_{\bullet}:P_{\bullet}\rar P_{\bullet}$ of $f$ to $P_{\bullet}$. 
Given $n\geq 0$ and $\phi \in \hom{B}(P_n,W)$, there is a $B$-linear map $f_n^{\#}(\phi):P_n\rar W$ defined by setting 
$$f_n^{\#}(\phi)(p)=f(\phi(p))- \phi(f_n(p)),$$ 
for $p\in P_n$.
The resulting morphism 
$$f_n^{\#}:\hom{B}(P_n,W)\rar \hom{B}(P_n,W)$$
is an endomorphism of the complex $\hom{B}(P_{\bullet},W)$. In fact, the induced map on cohomology
$$\Delta_{f,P_{\bullet}}^{\bullet}: \coho(\hom{B}(P_{\bullet},W))\rar\coho(\hom{B}(P_{\bullet},W))$$
only depends on $f$ and not on the choice of the lifting $f_{\bullet}$:

\begin{thm}\label{thm:Delta}\cite[Theorem A]{S}
	Let $W$ be a left $B$-module and $f$ be a $\delta$-operator on~$W$. There is a canonical morphism
	$$\Delta_f^{\bullet}:\ext{B}^{\bullet}(W,W)\rar \ext{B}^{\bullet}(W,W)$$
	of graded vector spaces,
	such that for every projective resolution $\epsilon:P_{\bullet}\srar W$ and each $\delta$-lifting $f_{\bullet}:P_{\bullet}\rar P_{\bullet}$ of $f$ to $P_{\bullet}$, the following diagram commutes:
	$$\begin{tikzcd}
		\coho(\hom{B}(P_{\bullet},W))\arrow{r}{\Delta_{f,P_{\bullet}}^{\bullet}} \arrow{d}{\cong} &\coho(\hom{B}(P_{\bullet},W))\arrow{d}{\cong}\\
		\ext{B}^{\bullet}(W,W)\arrow{r}{\Delta_f^{\bullet}} &\ext{B}^{\bullet}(W,W).
	\end{tikzcd}$$
\end{thm}

In what follows, let us consider an algebra $\alg$ and take $B=\alg^e$ and $W=A$ as a left $\alg^e$-module. In particular, $\alg$ can be taken to be the special biserial algebra $\alg_N$ defined in~Section~\ref{sec:intro}. 
It is well known that any element $f\in \hh^1(\alg)$ can be represented by a derivation $f:\alg \rar \alg$. 
We can use the above results to compute the Gerstenhaber bracket $[f, -]$ on $\hh^{\bullet}(\alg)$. 
The first step is to construct a derivation on $\alg^e$,
$$f^{e}:=f\T\id_\alg+\id_\alg\T f:\alg^e\longrightarrow \alg^e,$$ and note that $f:\alg\rar \alg$ is an $f^{e}$-operator on the $\alg^e$-module~$\alg$. 
We can thus consider the $f^{e}$-lifting of $f$ to the bar resolution $\barcx(\alg)$ of~$\alg$ and get a morphism of complexes 
$$f_{\bullet}^{\#}:\hom{\alg^e}(\barcx(\alg),\alg )\rar \hom{\alg^e}(\barcx(\alg),\alg ).$$

\begin{lemma}\cite[\S2.2]{S}.\label{lem:SA}
The morphism $f_{\bullet}^{\#}$ describes the action of the bracket $[f, -]$ on the complex $\hom{\alg^e}(\barcx(\alg),\alg)$.
In particular, the Gerstenhaber bracket $[f, -]$ on $\hh^{\bullet}(\alg)$ is given by 
$$\Delta_{f,\barcx}^{\bullet}:\hh^{\bullet}(\alg)\longrightarrow\hh^{\bullet}(\alg).$$
\end{lemma}

\begin{rem}\label{rem:SA} In later sections, we will make use of the following two general observations:
\begin{enumerate}
  \item By~Theorem~\ref{thm:Delta}, we can use any projective $A^{e}$-resolution $P_{\bullet}$ of $\alg$ to compute the bracket $[f,-]$, provided that we are able construct an $f^e$-lifting of $f$ to~$P_{\bullet}$. 
  \item Let $A$ be a $\ZZ$-graded $\ground$-algebra with $\ZZ$-grading $\mathsf{deg}$, and denote the induced grading on $A^e$ also by $\mathsf{deg}$. Write $\delta_{\mathsf{deg}}$ for the Eulerian derivation on $A$ defined by setting $\delta_{\mathsf{deg}}(a)=\mathsf{deg}(a)a$ for any homogeneous element $a\in A$.
Suppose $P_\bullet \srar A$ is a graded projective resolution of $A$ by $A^e$-modules; that is, every $P_n$ is a graded $A^e$-module and the differential maps preserve the grading. Then there is a $(\delta_{\mathsf{deg}})^e$-lifting $(\delta_{\mathsf{deg}})_{\bullet}$ of $\delta_{\mathsf{deg}}$, defined by setting 
$$ (\delta_{\mathsf{deg}})_n(p)\ =\ \mathsf{deg}_n(p)p$$
for homogeneous elements $p\in P_n$, where we denote the grading on $P_n$ by $\mathsf{deg}_n$.
\end{enumerate}
\end{rem}



\section{A minimal projective bimodule resolution for $\alg$}
\label{sec:resolution}

\begin{notation}
Throughout the paper, we fix integers $N\geq 1$ and $m\geq 3$. We consider a field $\ground$ and assume the characteristic of $\ground$ does not divide $2$, $N$ or~$m$. 
We let $$A:=\alg_N=\ground\quiv/ \brac {a_i a_{i+1},\ \ol{a}_{i+1}\ol{a}_i,\ (a_i\ol{a}_i)^N - (\ol{a}_{i-1}a_{i-1})^N \mid i \in\ZZ/m\ZZ},$$ 
where $\quiv$ is the quiver described in the~Introduction (Figure~\ref{fig:quiver}).  
Note that we identify the labeling set of the vertices of $\quiv$ with $\ZZ/m\ZZ$. 
\end{notation}

In \cite{ST}, Snashall and Taillefer construct a minimal projective $A^{e}$-resolution $(P_{\bullet},d_{\bullet})$ for $\alg$, and use it to compute the Hochschild cohomology $\hh^{\bullet}(\alg)$. In what follows, we describe~$P_{\bullet}$. 

For every $n\geq 0$, $i\in\ZZ/m\ZZ$ and $0\leq r\leq n$, we define
an element $\g{r,i}{n}\in \ground\quiv$. We set $\g{0,i}{0}=e_i$, the trivial path at vertex $i$, and for $n\geq 1$ we define
$$\g{r,i}{n}=
\begin{cases}
\g{r,i}{n-1}a  + (-1)^n \g{r-1,i}{n-1}\ol{a}(a\ol{a})^{N-1} 
& \mbox{ if } n-2r>0, \\
\g{r,i}{n-1}a(\ol{a}a)^{N-1}  + (-1)^n \g{r-1,i}{n-1}\ol{a} 
& \mbox{ if } n-2r<0, \\
\g{r,i}{n-1}a(\ol{a}a)^{N-1}  + \g{r-1,i}{n-1}\ol{a}(a\ol{a})^{N-1} 
& \mbox{ if } n=2r,
\end{cases}$$
where $\g{-1,i}{n}=\g{n,i}{n-1}=0$ by convention. In the above formulas, the indices of the $a$ and $\ol{a}$ arrows are chosen uniquely such that $g^n_{r,i}$ is nonzero in $\ground\quiv$. We also write $\bas{n}:=\{\g{r,i}{n}\mid i\in\ZZ/m\ZZ,\ 0\leq r\leq n\}$. In particular, $\ground \bas{0}=E$, the subalgebra of $\ground\quiv$ generated by all the trivial paths $e_i$'s, and
$$\begin{array}{lcl} 
\bas{1} & = & \{a_i,\ -\ol{a}_i\mid i\in\ZZ/m\ZZ \}, \\ 
\bas{2} & = & \{a_i a_{i+1},\ -\ol{a}_{i+1}\ol{a}_{i},\ (a_i\ol{a}_i)^N - (\ol{a}_{i-1}a_{i-1})^N\mid i\in\ZZ/m\ZZ \} .  \end{array}$$
Now, $\ground \bas{n}$ is an $E^e$-module for every $n\geq 0$, and we can consider the $A^{e}$-module
$$P_n:=\alg\otimes_E \ground \bas{n}\otimes_E \alg.$$ 
As usual, the augmentation map $\epsilon:P_0=\alg\otimes_E \alg\rar \alg$ is induced by multiplication on~$\alg$. 
For $n\geq 1$, the differential
$d_n:P_n\rar P_{n-1}$ is the $\alg^e$-linear map that sends $1\T\g{r,i}{n}\T 1$ to
$$\begin{cases}
1 \T \g{r,i}{n-1}\T a  
+ (-1)^{n+r} a \T \g{r,i+1}{n-1}\T 1 & \\
\quad + (-1)^{n+r} \ol{a}(a\ol{a})^{N-1}\T \g{r-1,i-1}{n-1}\T 1 
+ (-1)^{n} 1 \T \g{r-1,i}{n-1}\T \ol{a}(a\ol{a})^{N-1} 
& \mbox{if } n-2r>0, \\
\ \\
1 \T \g{r,i}{n-1}\T a(\ol{a}a)^{N-1}  
+ (-1)^{n+r} a(\ol{a}a)^{N-1} \T \g{r,i+1}{n-1}\T 1 & \\
\quad + (-1)^{n+r} \ol{a}\T \g{r-1,i-1}{n-1}\T 1
+ (-1)^{n} 1 \T \g{r-1,i}{n-1}\T \ol{a} 
& \mbox{if } n-2r<0,\\
\ \\
\sum_{k=0}^{N-1} \; \Big(
(\ol{a}a)^{k} \T\g{r,i}{n-1}\T a (\ol{a}a)^{N-1-k}
+(-1)^{\frac{n}{2}}(a\ol{a})^{k}a\T\g{r,i+1}{n-1}\T (a\ol{a})^{N-1-k} & \\
 
\quad + 
(-1)^{\frac{n}{2}}(\ol{a}a)^{k}\ol{a}\T\g{r-1,i-1}{n-1}\T (\ol{a}a)^{N-1-k}  
+(a\ol{a})^{k}\T \g{r-1,i}{n-1}\T \ol{a}(a\ol{a})^{N-1-k} \Big)
& \mbox{if } n=2r.  \end{cases}$$

In particular, the differential $d_1:P_1\rar P_0=A\T_E A$ is defined by sending
$$\left\lbrace \begin{array}{lll}
	1\T\g{0,i}{1}\T 1=1\T a_i\T 1 & \mbox{to} & 
	1\T a_i-a_i\T 1,\\
	1\T\g{1,i}{1}\T 1=-1\T \ol{a}_{i-1}\T 1 & \mbox{to} & 
	-1\T \ol{a}_{i-1}+\ol{a}_{i-1}\T 1.
	\end{array}\right. $$

\begin{thm}\cite[Theorem 1.6]{ST}
	Let $N\geq 1$. The complex $(P_{\bullet},d_{\bullet})$ is a minimal projective resolution  for $\alg$ as an $\alg^e$-module. 
\end{thm}

\begin{rem}\label{rem:N1}
We observe that when $N=1$, every $\g{r,i}{n} \in \ground\quiv$ is a $\ground$-linear combination of all the paths $p$ of length $n$ that start at the vertex $i$, such that $p$ contains exactly $r$ arrows of the form $\ol a$ (see Figure~2).
In this case, the differential $d_n:P_n\rar P_{n-1}$ maps $1\T\g{r,i}{n}\T 1$ to
$$1\T \g{r,i}{n-1}\T a_{i+n-2r-1}
+(-1)^{n+r} a_i\T \g{r,i+1}{n-1}\T 1
+(-1)^{n+r} \ol{a}_{i-1}\T \g{r-1,i-1}{n-1}\T 1
+(-1)^n 1\T \g{r-1,i}{n-1}\T \ol{a}_{i+n-2r}.$$
\end{rem}

\begin{rem}
To better illustrate the notation $\g{r,i}{n}$ in the resolution $P_{\bullet}$ given above, we provide the terms appearing in $\g{1,2}{4}$ when $N=1$ and when $N=2$ in Figures~2 and 3, respectively.
\end{rem}

\begin{figure}[H]\label{fig:paths}
\begin{center}
\begin{tikzpicture}[scale=1]
\begin{scope}[xshift=-6.2cm, scale=0.55]
\path[use as bounding box] (-3.5,-3.5) rectangle (3.5,2.5);

\foreach \x in {6,7,0,1,2,3,4,5}
\fill ({90-\x*45}:2.5cm) circle (1.2mm);

\foreach \x in {5,6,7,0,1,2,3,4}
\draw[thick, ->] ({85-\x*45}:2.65cm) arc ({85-\x*45}:{50-\x*45}:2.65cm);

\foreach \x in {5,6,7,0,1,2,3,4}
\draw[thick, <-] ({85-\x*45}:2.35cm) arc ({85-\x*45}:{52-\x*45}:2.35cm);

\foreach \x in {5,6,7,0,1,2,3,4}
\draw ({90-\x*45}:2cm) node {$\x$};

\draw[thick, red] ({85-2*45}:3cm) arc ({85-2*45}:{50-4*45}:3cm);
\draw[thick, red, <-] ({85-4*45}:3.3cm) arc ({85-4*45}:{50-4*45}:3.3cm);
\draw[thick, red] ({50-4*45}:3.3cm) arc ({410-4*45}:{230-4*45}:0.15cm);

\draw[red] (270:4.4cm) node {$a_2 a_3 a_4 \ol{a}_4$};

\end{scope}

\begin{scope}[xshift=-2.0cm, scale=0.55]
\path[use as bounding box] (-3.5,-3.8) rectangle (3.5,2.5);

\foreach \x in {6,7,0,1,2,3,4,5}
\fill ({90-\x*45}:2.5cm) circle (1.2mm);

\foreach \x in {5,6,7,0,1,2,3,4}
\draw[thick, ->] ({85-\x*45}:2.65cm) arc ({85-\x*45}:{50-\x*45}:2.65cm);

\foreach \x in {5,6,7,0,1,2,3,4}
\draw[thick, <-] ({85-\x*45}:2.35cm) arc ({85-\x*45}:{52-\x*45}:2.35cm);

\foreach \x in {5,6,7,0,1,2,3,4}
\draw ({90-\x*45}:2cm) node {$\x$};

\draw[thick, red] ({85-2*45}:3cm) arc ({85-2*45}:{50-3*45}:3cm);
\draw[thick, red] ({85-3*45}:3.3cm) arc ({85-3*45}:{50-3*45}:3.3cm);
\draw[thick, red, ->] ({85-3*45}:3.6cm) arc ({85-3*45}:{50-3*45}:3.6cm);

\draw[thick, red] ({50-3*45}:3.3cm) arc ({410-3*45}:{230-3*45}:0.15cm);
\draw[thick, red] ({85-3*45}:3.6cm) arc ({85-3*45}:{265-3*45}:0.15cm);

\draw[red] (270:4.4cm) node {$a_2 a_3 \ol{a}_3 a_3 $};

\end{scope}
\begin{scope}[xshift=2.0cm, scale=0.55]
\path[use as bounding box] (-3.5,-3.5) rectangle (3.5,2.5);

\foreach \x in {6,7,0,1,2,3,4,5}
\fill ({90-\x*45}:2.5cm) circle (1.2mm);

\foreach \x in {5,6,7,0,1,2,3,4}
\draw[thick, ->] ({85-\x*45}:2.65cm) arc ({85-\x*45}:{50-\x*45}:2.65cm);

\foreach \x in {5,6,7,0,1,2,3,4}
\draw[thick, <-] ({85-\x*45}:2.35cm) arc ({85-\x*45}:{52-\x*45}:2.35cm);

\foreach \x in {5,6,7,0,1,2,3,4}
\draw ({90-\x*45}:2cm) node {$\x$};

\draw[thick, red] ({85-2*45}:3cm) arc ({85-2*45}:{50-2*45}:3cm);
\draw[thick, red] ({85-2*45}:3.3cm) arc ({85-2*45}:{50-2*45}:3.3cm);
\draw[thick, red, ->] ({85-2*45}:3.6cm) arc ({85-2*45}:{50-3*45}:3.6cm);

\draw[thick, red] ({50-2*45}:3.3cm) arc ({410-2*45}:{230-2*45}:0.15cm);
\draw[thick, red] ({85-2*45}:3.6cm) arc ({85-2*45}:{265-2*45}:0.15cm);

\draw[red] (270:4.4cm) node {$a_2 \ol{a}_2 a_2 a_3 $};

\end{scope}
\begin{scope}[xshift=6.2cm, scale=0.55]
\path[use as bounding box] (-3.5,-3.5) rectangle (3.5,2.5);

\foreach \x in {6,7,0,1,2,3,4,5}
\fill ({90-\x*45}:2.5cm) circle (1.2mm);

\foreach \x in {5,6,7,0,1,2,3,4}
\draw[thick, ->] ({85-\x*45}:2.65cm) arc ({85-\x*45}:{50-\x*45}:2.65cm);

\foreach \x in {5,6,7,0,1,2,3,4}
\draw[thick, <-] ({85-\x*45}:2.35cm) arc ({85-\x*45}:{52-\x*45}:2.35cm);

\foreach \x in {5,6,7,0,1,2,3,4}
\draw ({90-\x*45}:2cm) node {$\x$};

\draw[thick, red] ({85-1*45}:3cm) arc ({85-1*45}:{50-1*45}:3cm);
\draw[thick, red, ->] ({85-1*45}:3.3cm) arc ({85-1*45}:{50-3*45}:3.3cm);

\draw[thick, red] ({85-1*45}:3.3cm) arc ({85-1*45}:{265-1*45}:0.15cm);

\draw[red] (270:4.4cm) node {$\ol{a}_1 a_1 a_2 a_3 $};

\end{scope}

\end{tikzpicture}
\end{center} \vskip 7pt
\caption{The paths that appear in $\g{1,2}{4}= a_2 a_3 a_4 \ol{a}_4 - a_2 a_3 \ol{a}_3 a_3 + a_2\ol{a}_2 a_2 a_3 - \ol{a}_1 a_1 a_2 a_3\in \ground\quiv$ when $N=1$, $m=8$.}
\end{figure}

\begin{figure}[H]\label{fig:pathsN}
\begin{center}
\begin{tikzpicture}[scale=1]
\begin{scope}[xshift=-6.2cm, scale=0.55]
\path[use as bounding box] (-3.5,-3.5) rectangle (3.5,2.5);

\foreach \x in {6,7,0,1,2,3,4,5}
\fill ({90-\x*45}:2.5cm) circle (1.2mm);

\foreach \x in {5,6,7,0,1,2,3,4}
\draw[thick, ->] ({85-\x*45}:2.65cm) arc ({85-\x*45}:{50-\x*45}:2.65cm);

\foreach \x in {5,6,7,0,1,2,3,4}
\draw[thick, <-] ({85-\x*45}:2.35cm) arc ({85-\x*45}:{52-\x*45}:2.35cm);

\foreach \x in {5,6,7,0,1,2,3,4}
\draw ({90-\x*45}:2cm) node {$\x$};

\draw[thick, red] ({85-2*45}:3cm) arc ({85-2*45}:{50-4*45}:3cm);
\draw[thick, red] ({85-4*45}:3.3cm) arc ({85-4*45}:{50-4*45}:3.3cm);

\draw[thick, red, dotted] ({85-4*45}:3.6cm) arc ({85-4*45}:{50-4*45}:3.6cm);
\draw[thick, red, dotted, <-] ({85-4*45}:3.9cm) arc ({85-4*45}:{50-4*45}:3.9cm);

\draw[thick, red] ({50-4*45}:3.3cm) arc ({410-4*45}:{230-4*45}:0.15cm);
\draw[thick, red, dotted] ({50-4*45}:3.9cm) arc ({410-4*45}:{230-4*45}:0.15cm);
\draw[thick, red, dotted] ({85-4*45}:3.6cm) arc ({85-4*45}:{265-4*45}:0.15cm);

\draw[red] (270:4.9cm) node {$a_2 a_3 a_4 \ol{a}_4 (a_4 \ol{a}_4)$};

\end{scope}

\begin{scope}[xshift=-2.0cm, scale=0.55]
\path[use as bounding box] (-3.5,-3.8) rectangle (3.5,2.5);

\foreach \x in {6,7,0,1,2,3,4,5}
\fill ({90-\x*45}:2.5cm) circle (1.2mm);

\foreach \x in {5,6,7,0,1,2,3,4}
\draw[thick, ->] ({85-\x*45}:2.65cm) arc ({85-\x*45}:{50-\x*45}:2.65cm);

\foreach \x in {5,6,7,0,1,2,3,4}
\draw[thick, <-] ({85-\x*45}:2.35cm) arc ({85-\x*45}:{52-\x*45}:2.35cm);

\foreach \x in {5,6,7,0,1,2,3,4}
\draw ({90-\x*45}:2cm) node {$\x$};

\draw[thick, red] ({85-2*45}:3cm) arc ({85-2*45}:{50-3*45}:3cm);
\draw[thick, red] ({85-3*45}:3.3cm) arc ({85-3*45}:{50-3*45}:3.3cm);
\draw[thick, red, dotted] ({85-3*45}:3.6cm) arc ({85-3*45}:{50-3*45}:3.6cm);
\draw[thick, red, dotted] ({85-3*45}:3.9cm) arc ({85-3*45}:{50-3*45}:3.9cm);

\draw[thick, red, ->] ({85-3*45}:4.2cm) arc ({85-3*45}:{50-3*45}:4.2cm);

\draw[thick, red] ({50-3*45}:3.3cm) arc ({410-3*45}:{230-3*45}:0.15cm);
\draw[thick, red] ({85-3*45}:3.6cm) arc ({85-3*45}:{265-3*45}:0.15cm);
\draw[thick, red, dotted] ({50-3*45}:3.9cm) arc ({410-3*45}:{230-3*45}:0.15cm);
\draw[thick, red, dotted] ({85-3*45}:4.2cm) arc ({85-3*45}:{265-3*45}:0.15cm);

\draw[red] (270:4.9cm) node {$a_2 a_3 \ol{a}_3 (a_3 \ol{a}_3)a_3 $};

\end{scope}
\begin{scope}[xshift=2.0cm, scale=0.55]
\path[use as bounding box] (-3.5,-3.5) rectangle (3.5,2.5);

\foreach \x in {6,7,0,1,2,3,4,5}
\fill ({90-\x*45}:2.5cm) circle (1.2mm);

\foreach \x in {5,6,7,0,1,2,3,4}
\draw[thick, ->] ({85-\x*45}:2.65cm) arc ({85-\x*45}:{50-\x*45}:2.65cm);

\foreach \x in {5,6,7,0,1,2,3,4}
\draw[thick, <-] ({85-\x*45}:2.35cm) arc ({85-\x*45}:{52-\x*45}:2.35cm);

\foreach \x in {5,6,7,0,1,2,3,4}
\draw ({90-\x*45}:2cm) node {$\x$};

\draw[thick, red] ({85-2*45}:3cm) arc ({85-2*45}:{50-2*45}:3cm);
\draw[thick, red] ({85-2*45}:3.3cm) arc ({85-2*45}:{50-2*45}:3.3cm);
\draw[thick, red, dotted] ({85-2*45}:3.6cm) arc ({85-2*45}:{50-2*45}:3.6cm);
\draw[thick, red, dotted] ({85-2*45}:3.9cm) arc ({85-2*45}:{50-2*45}:3.9cm);

\draw[thick, red, ->] ({85-2*45}:4.2cm) arc ({85-2*45}:{50-3*45}:4.2cm);

\draw[thick, red] ({50-2*45}:3.3cm) arc ({410-2*45}:{230-2*45}:0.15cm);
\draw[thick, red] ({85-2*45}:3.6cm) arc ({85-2*45}:{265-2*45}:0.15cm);
\draw[thick, red, dotted] ({50-2*45}:3.9cm) arc ({410-2*45}:{230-2*45}:0.15cm);
\draw[thick, red, dotted] ({85-2*45}:4.2cm) arc ({85-2*45}:{265-2*45}:0.15cm);

\draw[red] (270:4.9cm) node {$a_2 \ol{a}_2 (a_2\ol{a}_2) a_2 a_3 $};

\end{scope}
\begin{scope}[xshift=6.2cm, scale=0.55]
\path[use as bounding box] (-3.5,-3.5) rectangle (3.5,2.5);
\foreach \x in {6,7,0,1,2,3,4,5}
\fill ({90-\x*45}:2.5cm) circle (1.2mm);

\foreach \x in {5,6,7,0,1,2,3,4}
\draw[thick, ->] ({85-\x*45}:2.65cm) arc ({85-\x*45}:{50-\x*45}:2.65cm);

\foreach \x in {5,6,7,0,1,2,3,4}
\draw[thick, <-] ({85-\x*45}:2.35cm) arc ({85-\x*45}:{52-\x*45}:2.35cm);

\foreach \x in {5,6,7,0,1,2,3,4}
\draw ({90-\x*45}:2cm) node {$\x$};

\draw[thick, red] ({85-1*45}:3cm) arc ({85-1*45}:{50-1*45}:3cm);
\draw[thick, red] ({85-1*45}:3.3cm) arc ({85-1*45}:{50-1*45}:3.3cm);
\draw[thick, red, dotted] ({85-1*45}:3.6cm) arc ({85-1*45}:{50-1*45}:3.6cm);
\draw[thick, red, dotted] ({85-1*45}:3.9cm) arc ({85-1*45}:{50-1*45}:3.9cm);

\draw[thick, red, ->] ({50-1*45}:3.9cm) arc ({50-1*45}:{50-3*45}:3.9cm);

\draw[thick, red] ({85-1*45}:3.3cm) arc ({85-1*45}:{265-1*45}:0.15cm);
\draw[thick, red, dotted] ({50-1*45}:3.6cm) arc ({410-1*45}:{230-1*45}:0.15cm);
\draw[thick, red, dotted] ({85-1*45}:3.9cm) arc ({85-1*45}:{265-1*45}:0.15cm);

\draw[red] (270:4.9cm) node {$\ol{a}_1 a_1 (\ol{a}_1 a_1) a_2 a_3 $};

\end{scope}

\end{tikzpicture}
\end{center} \vskip 15pt

\caption{The paths that appear in $\g{1,2}{4}= a_2 a_3 a_4 \ol{a}_4 (a_4 \ol{a}_4)^{N-1} - a_2 a_3 \ol{a}_3 (a_3 \ol{a}_3)^{N-1} a_3 + a_2\ol{a}_2 (a_2\ol{a}_2)^{N-1} a_2 a_3 - \ol{a}_1 a_1 (\ol{a}_1 a_1)^{N-1} a_2 a_3\in \ground\quiv$ when $N=2$, $m=8$. All paths start at vertex $2$ and end at vertex $4$.}
\end{figure}



\section{The Hochschild cohomology ring for $\alg$} 
\label{cohomology}

\numberwithin{equation}{subsubsection}

In this section, we describe the Hochschild cohomology $\hh^{\bullet}(\alg)$ computed by Snashall and Taillefer. All results in this section appear in~\cite{ST}. For every $n\geq 0$, the $\ground$-module $\hh^{n}(\alg)$ is finite dimensional, and we describe its basis in terms of cocycles in $\hom{\alg^e}(P_n,\alg)$.  
As usual, we identify $\hom{\alg^e}(P_n,\alg)$ with $\hom{E^e}(\ground \bas{n},\alg)$, whose elements are the $\ground$-linear functions $\ground\bas{n}\rar\alg$ that map every $g\in \bas{n}$ to a linear combination of paths in $\alg$ parallel to~$g$, that is, paths that share the same source and target with $g$. For $g\in \bas{n}$ and $u\in A$ parallel to $g$, we write $(g\parallel u)\in  \hom{E^e}(\ground \bas{n},\alg)$ to denote the $\ground$-linear function that sends $g$ to $u$ and which maps all the other elements in $\bas{n}$ to zero.

\subsection{The center $\hh^{0}(\alg)$ of $\alg$.}\label{ss:0}

Writing $\varepsilon_i:=(a_i\ol{a}_i)^N\in \alg$ and $f_i:=(a_i\ol{a}_i + \ol{a}_i a_i)\in \alg$, the set 
$$\{1,\, \varepsilon_i,\, f_i^s \,\mid \, i\in\ZZ/m\ZZ, \, 1\leq s\leq N-1\}$$
is a basis for the $\ground$-module $\hh^{0}(\alg)$ under the usual identification of $\hh^0(\alg)$ with the center of $A$, see~\cite[Theorem 3.1]{ST}.

\subsection{The $\ground$-module $\hh^{n}(\alg)$ for $n\geq 1$ \cite[Propositions 4.1 and 5.1]{ST}} \label{ss:n}

For $m \geq 3$ and each $n \geq 1$, we write $n=pm +t$ for some integers $p,\,t$ with $p\geq 0$ and $0\leq t\leq m-1$. 

\subsubsection{When $m$ is even and $n$ is even,
the $\ground$-module $\hh^{n}(\alg)$ has basis:}\label{subsec:m even n even}

\begin{align*} 
\chi_{n,\alpha} &:= \sum_{i=0}^{m-1} \left( g^n_{\frac{n-\alpha m}{2},i}  \parallel (-1)^{\frac{n-\alpha m}{2}i}\,e_i \right) && \text{ for } -p \leq \alpha \leq p; \\
\pi_{n,\alpha}&:= \left( g^n_{\frac{n-\alpha m}{2},0}  \parallel (a_0\ol{a}_0)^N \right) && \text{ for } -p \leq \alpha \leq p; \\
F_{n,j,s}&:= \left( g^n_{\frac{n}{2},j}  \parallel (a_j\ol{a}_j)^s \right) + \left( g^n_{\frac{n}{2},j+1}  \parallel (-1)^{\frac{n}{2}}\,(\ol{a}_j a_j)^s \right) && \text{ for } j \in\ZZ/m\ZZ \text{ and } 1\leq s\leq N-1. 
\end{align*}

\subsubsection{When $m$ is even and $n$ is odd,
	the $\ground$-module $\hh^{n}(\alg)$ has basis:}\label{subsec:m even n odd}
\begin{align*} 
\varphi_{n,\gamma}&:=\sum_{i=0}^{m-1} \left( g^n_{\frac{n-\gamma m-1}{2},i}  \parallel (-1)^{\frac{n-\gamma m-1}{2}i} \, a_i (\ol{a}_i a_i)^{N-1} \right) && \text{ for } -p \leq \gamma < 0, \\
&&& \text{ and for }  \gamma =-(p+1) \text{ in case } t=m-1; \\
\varphi_{n,\gamma}&:=\sum_{i=0}^{m-1} \left( g^n_{\frac{n-\gamma m-1}{2},i}  \parallel (-1)^{\frac{n-\gamma m-1}{2}i} \, a_i \right) && \text{ for } 0 \leq \gamma \leq p; \\
\psi_{n,\beta}&:=\sum_{i=0}^{m-1} \left( g^n_{\frac{n-\beta m+1}{2},i}  \parallel (-1)^{\frac{n-\beta m-1}{2}i}  \,  \ol{a}_{i-1} (a_{i-1} \ol{a}_{i-1})^{N-1} \right) && \text{ for } 0 < \beta \leq p, \\
&&& \text{ and for } \beta =p+1 \text{ in case }  t=m-1; \\
\psi_{n,\beta}&:=\sum_{i=0}^{m-1} \left( g^n_{\frac{n-\beta m+1}{2},i}  \parallel (-1)^{\frac{n-\beta m-1}{2}i}  \,  \ol{a}_{i-1} \right) && \text{ for } -p \leq \beta \leq 0; \\
E_{n,j,s}&:= \left( g^n_{\frac{n-1}{2},j}  \parallel a_j (\ol{a}_j a_j)^s \right) && \text{ for } j \in\ZZ/m\ZZ \text{ and } 1\leq s\leq N-1.
\end{align*}

\subsubsection{When $m$ is odd and $n$ is even, the $\ground$-module $\hh^{n}(\alg)$ has basis:}
\label{subsec:m odd n even}
\begin{align*} 
\chi_{n,\delta} 
&:= \sum_{i=0}^{m-1} \left( g^n_{\frac{n-\delta m}{2},i}  \parallel e_i \right)
&& \text{ for } \delta
				=\begin{cases} p-2\alpha-1, 
				&\text{if $t$ is odd, $0 \leq \alpha < p$ and $\alpha + \frac{m-t}{2}$ is odd}, \\
				p-2\alpha, 
				&\text{if $t$ is even, $ 0 \leq \alpha \leq p$ and $\alpha + \frac{t}{2}$ is even}; \end{cases} \\
\pi_{n,\delta}
&:= \left( g^n_{\frac{n-\delta m}{2},0}  \parallel (a_0\ol{a}_0)^N \right)
&& \text{ for } \delta
				=\begin{cases} p-2\alpha-1, 
				& \text{if $t$ is odd, $0 \leq \alpha < p$ and $\alpha + \frac{m-t}{2}$ is even}, \\
                p-2\alpha, 
                & \text{if $t$ is even, $0 \leq \alpha \leq p$  and $\alpha + \frac{t}{2}$ is odd}; \end{cases} 
\end{align*}
\begin{align*}
F_{n,j,s}
&:= \left( g^n_{\frac{n}{2},j}  \parallel (a_j\ol{a}_j)^s \right) + \left( g^n_{\frac{n}{2},j+1}  \parallel (-1)^{\frac{n}{2}}\,(\ol{a}_j a_j)^s \right) 
&& \text{ for } j \in\ZZ/m\ZZ \text{ and } 1\leq s\leq N-1; \\
\varphi_{n,\sigma}
&:=\sum_{i=0}^{m-1} \left( g^n_{\frac{n-\sigma m-1}{2},i}  \parallel (a_i \ol{a}_i)^{N-1}a_i \right)
&& \text{ for }  t=m-1 \text{ and } \sigma=-(p+1); \\
\psi_{n,\tau}
&:=\sum_{i=0}^{m-1} \left( g^n_{\frac{n-\tau m+1}{2},i}  \parallel (\ol{a}_{i-1}a_{i-1})^{N-1}\ol{a}_{i-1} \right)
&& \text{ for } t=m-1 \text{ and } \tau=p+1.
\end{align*}

\begin{rem}\label{rem:indexchi}
We note that the value $\delta=0$ appears in the above index set for $\chi_{n,\delta}$ if and only if $n\equiv 0 \text{ (mod 4)}$. This follows by considering the possible values of $p,\ m$ and $t$ mod~$4$ for $\delta=0$.
\end{rem}

\subsubsection{When $m$ is odd and $n$ is odd, the $\ground$-module $\hh^{n}(\alg)$ has basis:} \label{subsec:m odd n odd}

\begin{align*} 
\varphi_{n,\sigma}
&:=\sum_{i=0}^{m-1} \left( g^n_{\frac{n-\sigma m-1}{2},i}  \parallel (a_i \ol{a}_i)^{N-1}a_i \right)&&\\
&\text{ for } \sigma 
				= \begin{cases} p-2\gamma
				& \text{if $t$ is odd, $\gamma \leq p<2 \gamma$ and $\gamma + \frac{t-1}{2}$ is even}, \\
				p-2\gamma-1
				& \text{if $t$ is even, $\gamma <p \leq 2\gamma$ and $\gamma + \frac{m+t-1}{2}$ is even, $t \neq m-1$, } \\
				p-2\gamma-1
				& \text{if $t=m-1$, $\gamma \leq p\leq 2\gamma$ and $\gamma$ is even}; \end{cases} &&\\
\varphi_{n,\sigma}
&:=\sum_{i=0}^{m-1} \left( g^n_{\frac{n-\sigma m-1}{2},i}  \parallel a_i \right)&&\\
&\text{ for } \sigma 
				= \begin{cases} p-2\gamma 
				& \text{if $t$ is odd, $0\leq 2\gamma \leq p$ and $\gamma + \frac{t-1}{2}$ is even}, \\
				p-2\gamma-1 
				& \text{if $t$ is even, $0\leq 2\gamma < p$ and $\gamma + \frac{m+t-1}{2}$ is even}; \end{cases} &&\\    
\psi_{n,\tau}
&:=\sum_{i=0}^{m-1} \left( g^n_{\frac{n-\tau m+1}{2},i}  \parallel (\ol{a}_{i-1}a_{i-1})^{N-1}\ol{a}_{i-1} \right)&&\\
&\text{ for } \tau 
				= \begin{cases} p-2\beta 
				& \text{ if $t$ is odd, $0\leq 2\beta <p$ and $\beta + \frac{t-1}{2}$ is even}, \\
				p-2\beta-1 
				& \text{ if $t$ is even, $0\leq 2\beta <p-1$ and $\beta + \frac{m+t-1}{2}$ is even, $t \neq m-1$,}\\
				p-2\beta-1 
				& \text{ if $t=m-1$, $-2\leq 2\beta < p-1$ and $\beta$ is even}; \end{cases} &&\\
\psi_{n,\tau}
&:=\sum_{i=0}^{m-1} \left( g^n_{\frac{n-\tau m+1}{2},i}  \parallel \ol{a}_{i-1} \right)&&\\
&\text{ for } \tau 
				= \begin{cases} p-2\beta 
				& \text{ if $t$ is odd, $\beta\leq p\leq 2\beta$ and $\beta + \frac{t-1}{2}$ is even,} \\
				p-2\beta-1 & 
				\text{ if $t$ is even, $\beta \leq p-1\leq 2\beta$ and $\beta + \frac{m+t-1}{2}$ is even}; \end{cases}&&\\                                                      
E_{n,j,s}
&:= \left( g^n_{\frac{n-1}{2},j}  \parallel a_j (\ol{a}_j a_j)^s \right)
 \text{ for } j \in\ZZ/m\ZZ \text{ and } 1\leq s\leq N-1;\\
\pi_{n,\delta}
&:= \left( g^n_{\frac{n-\delta m}{2},0}  \parallel (a_0\ol{a}_0)^N \right)
\text{ for } t=0 \text{ and } \delta = \pm p.
\end{align*}

\begin{rem}\label{rem:indexphi}
	We note that the value $\sigma=0$ appears in the above index set for $\varphi_{n,\sigma}$ if and only if $n\equiv 1  \text{ (mod 4)}$. Again, this follows by considering the possible values of $p,\ m$ and $t$ mod~$4$ for $\sigma=0$.
	Similarly, the value $\tau=0$ appears in the above index set for $\psi_{n,\tau}$ if and only if $n\equiv 1  \text{ (mod 4)}$.
\end{rem}

\subsection{For $m \geq 3$ even, the $\ground$-algebra $\hh^{\bullet}(\alg)$ has generators~\cite[Theorems 4.4 and 4.8]{ST}:} \label{m even ring}
\begin{align*}
& 1, \, \varepsilon_i,  && \text{ in degree } 0, \text{ for } N=1, \,i \in\ZZ/m\ZZ,\\
& 1, \, \varepsilon_i, \, f_i,  && \text{ in degree } 0, \text{ for } N>1, \,i \in\ZZ/m\ZZ, \\
& \varphi_{1,0}, \, \psi_{1,0} && \text{ in degree } 1,\\
& \chi_{2,0} && \text{ in degree } 2,\\
& \varphi_{m-1,-1}, \, \psi_{m-1,1} && \text{ in degree } m-1,\\
& \chi_{m,1}, \, \chi_{m,-1} && \text{ in degree } m.
\end{align*} 

We will compute the Gerstenhaber brackets for each of these generators in the next sections, so we describe them explicitly here:
\begin{align*}
\varepsilon_i&=(a_i\ol{a}_i)^N, & 
f_i=&a_i\ol{a}_i + \ol{a}_i a_i,\\
\varphi_{1,0} &= \sum_{i=0}^{m-1} \left(g^1_{0,i}  \parallel a_i\right), &
\psi_{1,0} =&  \sum_{i=0}^{m-1} \left(g^1_{1,i}  \parallel \ol{a}_{i-1}\right), \\
\chi_{2,0} &=  \sum_{i=0}^{m-1} \left(g^2_{1,i}  \parallel (-1)^ie_i\right), && \\
\varphi_{m-1,-1} &= \sum_{i=0}^{m-1} \left(g^{m-1}_{m-1,i}  \parallel (-1)^{i} \, a_i (\ol{a}_{i}a_{i})^{N-1}\right), &
\psi_{m-1,1} =& \sum_{i=0}^{m-1} \left(g^{m-1}_{0,i}  \parallel (-1)^i \, \ol{a}_{i-1} (a_{i-1} \ol{a}_{i-1})^{N-1}\right), \\
\chi_{m,1} &=  \sum_{i=0}^{m-1} \left(g^m_{0,i}  \parallel e_i\right), &
\chi_{m,-1} =&  \sum_{i=0}^{m-1} \left(g^m_{m,i}  \parallel e_i\right).
\end{align*}

\subsection{For $m \geq 3$ odd, the $\ground$-algebra $\hh^{\bullet}(\alg)$ has generators~\cite[Theorems 5.2 and 5.4]{ST}:} \label{m odd ring}

\begin{align*}
& 1, \, \varepsilon_i,  && \text{ in degree } 0, \text{ for } N=1, \,i \in\ZZ/m\ZZ,\\
& 1, \, \varepsilon_i, \, f_i,  && \text{ in degree } 0, \text{ for } N>1, \,i \in\ZZ/m\ZZ, \\
& \varphi_{1,0}, \, \psi_{1,0} && \text{ in degree } 1,\\
& F_{2,j,1} && \text{ in degree } 2, \text{ for } N>1, \,j \in\ZZ/m\ZZ,\\
& \chi_{4,0} && \text{ in degree } 4,\\
& \varphi_{m-1,-1}, \, \psi_{m-1,1} && \text{ in degree } m-1,\\
& \chi_{2m,2}, \, \chi_{2m,-2} && \text{ in degree } 2m.
\end{align*}

Explicitly, these generators are:
\begin{align*}
\varepsilon_i&=(a_i\ol{a}_i)^N, & 
f_i=&a_i\ol{a}_i + \ol{a}_i a_i,\\
\varphi_{1,0} &= \sum_{i=0}^{m-1} \left(g^1_{0,i}  \parallel a_i\right), &
\psi_{1,0} =& \sum_{i=0}^{m-1} \left(g^1_{1,i}  \parallel \ol{a}_{i-1}\right), \\
F_{2,j,1}&= \left( g^2_{1,j}  \parallel a_j\ol{a}_j \right) + \left( g^2_{1,j+1}  \parallel - \ol{a}_j a_j \right), &&\\
\chi_{4,0} &= \sum_{i=0}^{m-1} \left(g^4_{2,i}  \parallel e_i\right), && \\
\varphi_{m-1,-1} &= \sum_{i=0}^{m-1} \left(g^{m-1}_{m-1,i}  \parallel (a_i \ol{a}_i)^{N-1} a_i\right), & 
\psi_{m-1,1} =& \sum_{i=0}^{m-1} \left(g^{m-1}_{0,i}  \parallel ( \ol{a}_{i-1} a_{i-1})^{N-1}  \ol{a}_{i-1}\right), \\
\chi_{2m,2} &= \sum_{i=0}^{m-1} \left(g^{2m}_{0,i}  \parallel e_i\right), &
\chi_{2m,-2} =& \sum_{i=0}^{m-1} \left(g^{2m}_{2m,i}  \parallel e_i\right).
\end{align*}

We list several algebra relations on the generators for $m$ even and odd in  Method~\ref{methodc}. Some of the relations given there differ from the results in \cite[Theorem 4.8]{ST}.



\section{Gerstenhaber brackets with $\varphi_{1,0}$ and $\psi_{1,0}$}
\label{brackets HH1}

\numberwithin{equation}{section}

In this section, we compute the brackets with basis elements $\varphi_{1,0}$ and $\psi_{1,0}$ of $\hh^{1}(\alg)$, using Su\'{a}rez-\'{A}lvarez's approach as described in Section~\ref{sec:approach}. 
More precisely, we show that the basis for $\hh^n(\alg)$ described in~Section~\ref{cohomology} is an eigenbasis for the endomorphisms $[\varphi_{1,0},-]$ and $[\psi_{1,0},-]$.
We proceed as follows.

First we show that $\varphi_{1,0}$, $\psi_{1,0}$ correspond to Eulerian derivations on $\alg$ coming from gradings $d$, $\ol{d}$ on $\alg$, respectively. Then we check that the projective resolution $P_\bullet\srar \alg$ given in~Section~\ref{sec:resolution} is graded with respect to both gradings. 
By Remark~ \ref{rem:SA}, we can then describe the brackets $[\varphi_{1,0},-]$ and $[\psi_{1,0},-]$ on $\hh^{n}(\alg)$ for all $n\ge 0$ in terms of these gradings. 

Furthermore, we will show that the basis for $\hh^n(\alg)$ given in~Section~\ref{cohomology} consists of homogeneous elements with respect to both gradings. In this case $[\varphi_{1,0},-]$ and $[\psi_{1,0},-]$ act diagonally with eigenvalues given by the degree of the homogeneous basis elements.

Recall that
$$
\varphi_{1,0} = \sum_{i=0}^{m-1}(\g{0,i}{1} \parallel a_i)=\sum_{i=0}^{m-1}(a_i \parallel a_i)\quad\mbox{and} \quad
\psi_{1,0} = \sum_{i=0}^{m-1}(\g{1,i}{1} \parallel \ol{a}_{i-1})=\sum_{i=0}^{m-1}(-\ol{a}_{i} \parallel \ol{a}_{i}).
$$

Consider the grading $d$ on $\ground\quiv$ such that the arrows $a_i$ are in degree $1$ and the arrows $\ol{a}_i$ are in degree $0$.  Similarly, consider the grading $\ol{d}$ on $\ground\quiv$ such that the arrows $a_i$ are in degree $0$ and the arrows $\ol{a}_i$ are in degree $-1$. We will write $d(p)$ and $\ol{d}(p)$ for the degrees of the path $p\in\ground \quiv$ under the gradings $d$ and $\ol{d}$, respectively.
Moreover, since $\alg$ is defined as a quotient of $\ground \quiv$ by an ideal which is homogeneous under both gradings, we get two corresponding gradings $d$ and $\ol{d}$ on $\alg$. The algebra $\alg^e$ inherits these gradings as well, by setting $d(a\otimes a\p)=d(a)+d(a\p)$ and $\ol{d}(a\otimes a\p)=\ol{d}(a)+\ol{d}(a\p)$ for homogeneous elements $a,a\p\in\alg$.

Now, the elements $\varphi_{1,0}$ and $\psi_{1,0}$ in $\hh^1(\alg)$ correspond to the derivations $\delta_{d}, \delta_{\ol{d}}:\alg\rar\alg$ sending a path $p$ in $\alg$ to $d(p)p$ and $\ol{d}(p)p$, respectively. 
In other words, $\delta_{d}$ is the Eulerian derivation associated to the grading $d$ on $\alg$, while $\delta_{\ol{d}}$ is the Eulerian derivation associated to the grading $\ol{d}$ on $\alg$. 

Observe furthermore that:
\begin{lemma}\label{lem:deg}
The elements $g_{r,i}^n$ in $\ground\quiv$ are homogeneus with respect to the gradings $d$ and $\ol{d}$, with degree given by
\[ d( \g{r,i}{n} ) =
\begin{cases}
rN + n-2r  & \mbox{if } n-2r \geq 0\\
(n-r) N 
& \mbox{if } n-2r<0
\end{cases}\quad \text{and} \quad \ol{d} ( \g{r,i}{n} ) =
\begin{cases}
-rN  & \mbox{if } n-2r \geq 0\\
-(n-r) N +n-2r 
& \mbox{if } n-2r<0.
\end{cases}\]
Furthermore, the projective $\alg^e$-module $P_n= \alg\otimes_E \langle g_{r,i}^n\ |\ i\in\ZZ/m\ZZ,0\le r\le n\rangle\otimes_E \alg$ inherits both gradings (denoted again by $d$, $\ol{d}$) by setting
\begin{align*}
d(a\otimes g_{r,i}^n \otimes a\p)   &=d(a)+d(g_{r,i}^n)+ d(a\p),\\
\ol{d}(a\otimes g_{r,i}^n \otimes a\p)&=\ol{d}(a)+\ol{d}(g_{r,i}^n)+\ol{d}(a\p),
\end{align*}
for homogeneous elements $a,a\p\in \alg$. Then the minimal projective resolution $P_\bullet\srar \alg$ from Section \ref{sec:resolution} is a graded projective resolution with respect to these two gradings.
\end{lemma}
\begin{proof}
To compute the degree of $g_{r,i}^n$, recall the recursive definition of $g_{r,i}^n$ given in~Section~\ref{sec:resolution}, and use induction on $n$.
For the graded projective resolution, observe that the differential maps $d_n$, defined in Section~\ref{sec:resolution}, preserve the gradings: 
a case-by-case analysis shows that all the summands appearing in the image of the differential $d_n(1\otimes \g{r,i}{n}\otimes 1)$ are again of degree $d(\g{r,i}{n})$ and $\ol{d}(\g{r,i}{n})$, respectively. Here, we use the explicit degree formulas for $\g{r,i}{n}$.
\end{proof}

Now we can apply the observations from Remark \ref{rem:SA} to both $d$ and $\ol{d}$ and obtain the following result:

\begin{prop}\label{prop:lift}
The $(\delta_d)^e$-operator $(\delta_d)_n:P_n\rar P_n$ given by
$$(\delta_d)_n (1\T\g{r,i}{n}\T1) = d(\g{r,i}{n})(1\T\g{r,i}{n}\T1 )$$
defines a
$(\delta_d)^e$-lifting $(\delta_d)_{\bullet}$ of $\delta_d$ to $P_{\bullet}$. 
Similarly, the $(\delta_{\ol{d}})^{e}$-operator $(\delta_{\ol{d}})_n:P_n\rar P_n$ given by
$$(\delta_{\ol{d}})_n (1\T\g{r,i}{n}\T1) = \ol{d}(\g{r,i}{n})(1\T\g{r,i}{n}\T1 )$$
defines a
$(\delta_{\ol{d}})^e$-lifting $(\delta_{\ol{d}})_{\bullet}$ of $\delta_{\ol{d}}$ to $P_{\bullet}$.  
\end{prop}

\begin{rem}\label{rem:gradhh} 
Since $P_{\bullet}\srar A$ is a graded projective resolution of $A$ with respect to the grading $d$, the cohomology $\hh^n(\alg)$ inherits a grading $d$ for every $n\geq 0$ as usual. 
More precisely, the $\ground$-module $\hom{\alg^e}(P_n,\alg) \cong \hom{E^e}(\ground \bas{n},\alg)$ inherits a grading $d$ by setting	
$$
d((\g{r,i}{n}\parallel u))=d(u)-d(\g{r,i}{n}),
$$
for $u\in\alg$ homogeneous and parallel to $g^n_{r,i}$.
This grading on $\hom{\alg^e}(P_{n},\alg)$ behaves well with respect to the differential, so the cohomology $\hh^n(\alg)$ acquires a grading $d$ for every $n\geq 0$.
	
Similarly, since $P_{\bullet}\srar A$ is graded with respect to the grading $\ol{d}$, the $\ground$-module $\hom{\alg^e}(P_n,\alg)$ and the cohomology $\hh^n(\alg)$ inherit a grading $\ol{d}$ for every $n\geq 0$. 
We note that the basis for $\hh^n(\alg)$ given in~Section~\ref{cohomology} consists of homogenous elements with respect to both gradings. \end{rem}

The gradings on $\alg$ and on $\hh^0(\alg)$, considered as the center of $\alg$, agree. In this case we can compute the brackets of $\varphi_{1,0}$ and $\psi_{1,0}$ with generators of $\hh^0(\alg)$ directly as follows.
\begin{align*}
&[\varphi_{1,0},1]=0 
&& \text{and} &&
[\psi_{1,0},1]=0, \\
&[\varphi_{1,0},\varepsilon_i]=\delta_d((a_i\ol{a}_i)^N)=N(a_i\ol{a}_i)^N 
&& \text{and} &&
[\psi_{1,0},\varepsilon_i]=\delta_{\ol{d}}((a_i\ol{a}_i)^N)=-N(a_i\ol{a}_i)^N ,\\
&[\varphi_{1,0},f_i]=\delta_d(a_i\ol{a}_i+\ol{a}_ia_i)=a_i\ol{a}_i+\ol{a}_ia_i
&& \text{and} &&
[\psi_{1,0},f_i]=\delta_{\ol{d}}(a_i\ol{a}_i+\ol{a}_ia_i)=-(a_i\ol{a}_i+\ol{a}_ia_i).
\end{align*}

The following proposition is an immediate consequence of Remark \ref{rem:gradhh}. 

\begin{prop}\label{diagonal}
	Let $n\geq 0$. The elements $\varphi_{1,0}$ and $\psi_{1,0}$ in $\hh^1(\alg)$ act diagonally on~$\hh^n(\alg)$. 
	That is, the $\ground$-basis for $\hh^n(\alg)$, given in Sections~\ref{ss:0} and~\ref{ss:n}, is an eigenbasis for the endomorphisms
	$[\varphi_{1,0},-]$ and $[\psi_{1,0},-]$ of $\hh^n(\alg)$. The corresponding eigenvalues are given by the degree under $d$ and $\ol{d}$ respectively, as listed in Table~\ref{tab:diagonaleven} and~Table~\ref{tab:diagonalodd} below.
\end{prop}
\begin{proof}
For $u\in\alg$ homogeneous, using~Lemma~\ref{lem:SA} and Proposition~\ref{prop:lift}, we can compute
\begin{align*}
[\varphi_{1,0}, (\g{r,i}{n}\parallel u)] & =\Delta_{\varphi_{1,0},P_{\bullet}}^{n}((\g{r,i}{n}\parallel u))
 =\delta_d (\g{r,i}{n}\parallel u)- (\g{r,i}{n}\parallel u)(\delta_d)^n \\
&= d(u) (\g{r,i}{n}\parallel u) - d(\g{r,i}{n})(\g{r,i}{n}\parallel u)\\
& = d((\g{r,i}{n}\parallel u))(\g{r,i}{n}\parallel u)
\end{align*}
and
\begin{align*}
[\psi_{1,0}, (\g{r,i}{n}\parallel u)] &=\Delta_{\psi_{1,0},P_{\bullet}}^{n}((\g{r,i}{n}\parallel u))=\delta_{\ol{d}} (\g{r,i}{n}\parallel u)- (\g{r,i}{n}\parallel u)(\delta_{\ol{d}})^n \\
&= \ol{d}(u) (\g{r,i}{n}\parallel u) - \ol{d}(\g{r,i}{n})(\g{r,i}{n}\parallel u)\\
& = \ol{d}((\g{r,i}{n}\parallel u))(\g{r,i}{n}\parallel u).
\end{align*}
The result now follows from Remark~\ref{rem:gradhh}.
\end{proof}

Let us now illustrate this statement by some explicit computations:

\begin{rem}
The $[\varphi_{1,0},-]$-action on $(\g{r,i}{n}\parallel u)$ for a nonzero path $u\in \alg$ can be obtained as follows:
\begin{align*}
[\varphi_{1,0}, (\g{r,i}{n}\parallel u)]  = \begin{cases}
(d(u)- n+2r-rN ) (\g{r,i}{n}\parallel u) & \mbox{if } n-2r \geq 0\\
(d(u)- (n-r)N ) (\g{r,i}{n}\parallel u)
& \mbox{if } n-2r<0,
\end{cases} 
\end{align*}
where $d(u)$ is the number of clockwise arrows that appear in the  path $u$. Similarly,
\begin{align*}
[\psi_{1,0}, (\g{r,i}{n}\parallel u)] =\begin{cases}
	(\ol{d}(u)+ rN) (\g{r,i}{n}\parallel u) & \mbox{if } n-2r \geq 0\\
	(\ol{d}(u)+ (n-r) N  -n +2r)  ) (\g{r,i}{n}\parallel u)
	& \mbox{if } n-2r<0,
	\end{cases}
\end{align*}
where $-\ol{d}(u)$ is the number of counterclockwise arrows that appear in the nonzero path $u\in \alg$.
\end{rem}

\begin{rem} We observe the following:
\begin{enumerate} 
 \item If $f,g \in \{\varphi_{1,0}, \psi_{1,0}\}$ then $[f,g]=0$.
 \item The sum $\varphi_{1,0} + \psi_{1,0}$ is in the center of $\hh^1(\alg)$, as $[\varphi_{1,0} + \psi_{1,0},\hh^1(\alg)]=0$. Moreover, we compute
	\[ [\varphi_{1,0} + \psi_{1,0}, (\g{r,i}{n}\parallel u)] 
	=  \left( d(u) + \ol{d}(u)  - n+2r \right)  (\g{r,i}{n}\parallel u),\] and 
	\[ [\varphi_{1,0} - \psi_{1,0}, (\g{r,i}{n}\parallel u)] = 
	\begin{cases}
	 \left( d(u) - \ol{d}(u) - n+2r- 2r N  \right)   (\g{r,i}{n}\parallel u)               & \mbox{if } n-2r \geq 0 \\
	 \left( d(u) -  \ol{d}(u) -2nN+2rN+n-2r  \right)  (\g{r,i}{n}\parallel u)	& \mbox{if } n-2r<0.
	\end{cases} \]
\end{enumerate}
\end{rem}

\begin{rem} In the case $N=1$, the authors of \cite{GTW} define three gradings on $\alg$:
the first one is obtained by setting the degree of $a_i$ equal to
$1$ for all $i$, while the degree of any other generator is zero;
the second grading is simply the path length grading. They note that with
the given relations, any path is at most of length $2$ in the algebra, hence they consider this second grading as a $\mathbb{Z}/2\mathbb{Z}$
grading. The third grading is defined by setting the degree of $a_{m-1}$ equal to $1$, the degree of $\ol{a}_{m-1}$ equal to $-1$, and the degree of any other arrow is zero. The first grading mentioned in \cite{GTW} is in fact $d$, while the second one corresponds to $d-\overline{d}$. The third grading corresponds to a Eulerian derivation which is a linear combination of the preceding two and coboundaries.
\end{rem}

In the following Table~\ref{tab:diagonaleven} and~Table~\ref{tab:diagonalodd}, we record the eigenvalues of the endomorphisms $[\varphi_{1,0},\ ]$, $[\psi_{1,0}, \ ]$, $[\varphi_{1,0} + \psi_{1,0}, \ ]$ and $[\varphi_{1,0} - \psi_{1,0},\ ]$ of $\hh^n(\alg)$. The basis elements of $\hh^n(\alg)$, with explicit conditions on the indices, are given in Section~\ref{ss:n}. We let $i \in \mathbb{Z}/m\mathbb{Z}$, $1 \leq s \leq N-1$, and recall that we write $n=pm +t$ with $p\geq 0$ and $0\leq t\leq m-1$.


\newcommand\MyLBrace[2]{%
	\text{#2}\left\{\rule{0pt}{#1}\right.}

\begin{landscape}
	\renewcommand{\arraystretch}{1.3}
\FloatBarrier
\begin{table}[!htp]
	
\begin{center}
		\caption{Eigenvalues when $m \geq 3$ is \textbf{even}.}  		  	\label{tab:diagonaleven}
		\vspace{0.1cm}
		\begin{tabular}{c@{}l}
			$\begin{array}{r}
			\vspace{0.4ex}\\
			\MyLBrace{5ex}{$n=0$} 
			\vspace{3ex}\\ 
			\vspace{3ex}
			\MyLBrace{9.5ex}{$n$ even} \\
			\MyLBrace{10ex}{$n$ odd} 
			\end{array}$
			&
		\begin{tabular}{| p{0.9cm} p{2.5cm} | p{3.6cm} | p{3.6cm} | p{2.5cm} | p{4.2cm} |}

			\hline
			$\mathbf{\hh^n}$ & 
			& $\mathbf{[\varphi_{1,0},\ \ ]}$  
			& $\mathbf{[\psi_{1,0},\ \ ]}$ 
			& $\mathbf{[\varphi_{1,0}+\psi_{1,0},\ \ ]}$
			& $\mathbf{[\varphi_{1,0}-\psi_{1,0},\ \ ]}$
			\\ \hline   \hline  
			
			$1$ &
			& $0$ 
			& $0$ 
			& $0$
			& $0$
			\\ \hline

			$\varepsilon_i$ 
			& 
			& $N$ 
			& $-N$ 
			& $0$
			& $2N$
			\\ \hline
			
			$f_i^s$  
			& 
			& $s$ 
			& $-s$
			& $0$
			& $2s$
			\\ \hline \hline \hline 
			
			$\chi_{n,\alpha}$  
			& $\begin{cases}
					\mbox{with } \alpha \geq 0  \\
					\mbox{with } \alpha<0 \end{cases} $
			& $\begin{cases}
					-\alpha m- \left(\frac{n-\alpha m}{2}\right)N  \\
					-\left(\frac{n+\alpha m}{2}\right) N  \end{cases} $
			& $\begin{cases}
					\left(\frac{n-\alpha m}{2}\right)N   \\
					- \alpha m + \left(\frac{n+\alpha m}{2}\right) N \end{cases}$
			& $-\alpha m$ 
			& $\begin{cases}
					-\alpha m -\left(n-\alpha m\right)N \\
					\alpha m -\left(n+\alpha m\right)N \end{cases}$
			\vspace{0.5em} \\  \hline

			$\pi_{n,\alpha}$ 
			& $\begin{cases}
					\mbox{with } \alpha \geq 0  \\
					\mbox{with } \alpha<0 \end{cases} $
			& $\begin{cases}
					-\alpha m- \left(\frac{n-\alpha m - 2}{2}\right)N  \\
					-\left(\frac{n+\alpha m-2}{2}\right) N  \end{cases}$  
			& $\begin{cases}
					\left(\frac{n-\alpha m-2}{2}\right)N    \\
					-\alpha m + \left(\frac{n+\alpha m-2}{2}\right) N  \end{cases}$
			& $-\alpha m$
			& $\begin{cases}
					-\alpha m -\left(n-\alpha m-2\right)N \\
					\alpha m -\left(n+\alpha m-2\right)N \end{cases}$
			\vspace{0.3em}\\  \hline
			
			$F_{n,i,s}$ 
			& 
			& $s-\frac{n}{2}N$ 
			& $\frac{n}{2}N - s $
			& $0$
			& $2s-nN$
			\\  \hline \hline \hline 
			
			$\varphi_{n,\gamma}$ 
			& $\begin{cases}
					\mbox{with } \gamma \geq 0 \\ 
					\mbox{with } \gamma < 0 \end{cases} $
			& $\begin{cases}
					-\gamma m-\left(\frac{n-\gamma m-1}{2}\right)N \\
					-\left(\frac{n+\gamma m-1}{2}\right)N \end{cases}$ 
			& $\begin{cases}
					\left(\frac{n-\gamma m-1}{2}\right)N \\
					-\gamma m + \left(\frac{n + \gamma m-1}{2}\right)N \end{cases}$
			& $-\gamma m$
			& $\begin{cases}
					-\gamma m -\left(n-\gamma m-1\right)N \\
					\gamma m -\left(n+\gamma m-1\right)N \end{cases}$
			\vspace{0.3em} \\  \hline
			
			$\psi_{n,\beta}$ 
			& $\begin{cases}
					\mbox{with } \beta >0\\ 
					\mbox{with } \beta \leq 0  \end{cases}$
			& $\begin{cases}
					-\beta m-\left(\frac{n-\beta m-1}{2}\right)N\\ 
					-\left(\frac{n+\beta m-1}{2}\right)N \end{cases}$
			& $\begin{cases}
					\left(\frac{n-\beta m-1}{2}\right)N \\
					-\beta m+\left(\frac{n+\beta m-1}{2}\right)N\end{cases}$
			& $-\beta m$
			& $\begin{cases}
					-\beta m -\left(n-\beta m-1\right)N \\
					\beta m -\left(n+\beta m-1\right)N \end{cases}$
			\vspace{0.3em} \\  \hline	
			
			$E_{n,i,s}$ 
			& 
			& $s-\frac{n-1}{2}N$ 
			& $\frac{n-1}{2}N - s $ 
			& $0$
			& $2s-(n-1)N$
			\\ \hline
		\end{tabular}
	\end{tabular}
	\end{center}
\end{table} 
\end{landscape}


\begin{landscape}
	\renewcommand{\arraystretch}{1.3}
\FloatBarrier
\begin{table}[!htp]
\begin{center}
		\caption{Eigenvalues when $m \geq 3$ is \textbf{odd}.} \label{tab:diagonalodd}
		\vspace{0.1cm}
		\begin{tabular}{c@{}l}
			$\begin{array}{r}
			\vspace{0.4ex}\\
			\MyLBrace{5ex}{$n=0$} 
			\vspace{3ex}\\ 
			\vspace{3ex}
			\MyLBrace{13ex}{$n$ even} \\
			\MyLBrace{13.6ex}{$n$ odd} 
			\end{array}$
						&
		\begin{tabular}{| p{0.9cm} p{3.6cm} | p{3.6cm} | p{3.6cm} | p{2.5cm} | p{4.2cm} |}
			
			\hline
			$\mathbf{\hh^n}$ & 
			& $\mathbf{[\varphi_{1,0},\ \ ]}$  
			& $\mathbf{[\psi_{1,0},\ \ ]}$ 
			& $\mathbf{[\varphi_{1,0}+\psi_{1,0},\ \ ]}$
			& $\mathbf{[\varphi_{1,0}-\psi_{1,0},\ \ ]}$
			\\ \hline   \hline
			 
			$1$ &
			& $0$ 
			& $0$ 
			& $0$
			& $0$
			\\ \hline
			
			$\varepsilon_i$ 
			&  
			& $N$ 
			& $-N$ 
			& $0$
			& $2N$
			\\ \hline
			
			$f_i^s$  
			& 
			& $s$ 
			& $-s$
			& $0$
			& $2s$
			\\ \hline \hline \hline 
			
			$\chi_{n,\delta}$  
			& $\begin{cases}
					\mbox{with } \delta \geq 0   \\
					\mbox{with } \delta<0 \end{cases} $
			& $\begin{cases}
					-\delta m- (\frac{n-\delta m}{2})N\\
					-(\frac{n+\delta m}{2}) N  \end{cases}$
			& $\begin{cases}
					(\frac{n-\delta m}{2})N   \\
					-\delta m + (\frac{n+\delta m}{2}) N \end{cases}$   
			& $-\delta m$
			& $\begin{cases}
					-\delta m- (n-\delta m)N   \\
					\delta m - (n+\delta m)N\end{cases}$   
			\vspace{0.5em}\\  \hline

			$\pi_{n,\delta}$  
			& $\begin{cases}
					\mbox{with } \delta \geq 0  \\
					\mbox{with } \delta<0 \end{cases} $
			& $\begin{cases}
					-\delta m- (\frac{n-\delta m - 2}{2})N\\
					-(\frac{n+\delta m-2}{2}) N  \end{cases}$  
			& $\begin{cases}
					(\frac{n-\delta m-2}{2})N    \\
					- \delta m + (\frac{n+\delta m-2}{2}) N  \end{cases}$ 
			& $-\delta m$
			& $\begin{cases}
					-\delta m- (n-\delta m-2)N   \\
					\delta m - (n+\delta m-2)N \end{cases}$ 
			\vspace{0.3em}\\  \hline
			
			$F_{n,i,s}$ 
			& 
			& $s-\frac{n}{2}N$ 
			& $\frac{n}{2}N - s $
			&  $0$
			& $2s-nN$
			\\  \hline
			
			$\varphi_{n,\sigma}$ 
			&
			& $N$
			& $-\sigma m -N $
			& $-\sigma m$
			& $2N+\sigma m$ 
			\\  \hline
			
			$\psi_{n,\tau}$
			& 
			& $-\tau m+N$
			& $-N$
			& $-\tau m$
			& $2N-\tau m$ 
			\\  \hline \hline \hline 
			
			$\pi_{n,\delta}$
			& $\begin{cases}
						\mbox{with } t=0, \ \delta= p   \\
						\mbox{with } t=0, \ \delta=- p \ \end{cases}$
			& $\begin{cases}
						-\delta m +N\\
						N  \end{cases}$  
			& $\begin{cases}
						-N    \\
						- \delta m -N   \end{cases}$ 
			& $-\delta m$
			& $\begin{cases}
						2N-\delta m    \\
						2N+\delta m   \end{cases}$
			\vspace{0.3em} \\  \hline
			
			$\varphi_{n,\sigma}$ 
			& $\begin{cases}
						\mbox{with } \sigma \geq 0 \\
						\mbox{with } \sigma < 0 \end{cases} $
			& $\begin{cases}
						-\sigma m-\left(\frac{n-\sigma m-1}{2}\right)N \\
						-\left(\frac{n+\sigma m-1}{2}\right)N \end{cases}$ 
			& $\begin{cases}
						\left(\frac{n-\sigma m-1}{2}\right)N \\
						-\sigma m + \left(\frac{n + \sigma m-1}{2}\right)N \end{cases}$ 
			& $-\sigma m$
			& $\begin{cases}
						-\sigma m - (n-\sigma m-1)N \\
						\sigma m- (n+\sigma m-1)N \end{cases}$ 
			\vspace{0.3em} \\  \hline
			
			$\psi_{n,\tau}$ 
			& $\begin{cases}
						\mbox{with } \tau >0\\ 
						\mbox{with } \tau \leq 0  \end{cases}$
			& $\begin{cases}
						-\tau m-\left( \frac{n-\tau m-1}{2}\right) N\\
						-\left(\frac{n+\tau m-1}{2}\right)N \end{cases}$ 
			& $\begin{cases}
						\left(\frac{n-\tau m-1}{2}\right)N \\ 
						-\tau m+\left(\frac{n+\tau m-1}{2}\right)N\end{cases}$ 
			& $-\tau m$
			& $\begin{cases}
						-\tau m- (n-\tau m-1)N   \\
						\tau m - (n+\tau m-1)N \end{cases}$ 					
			\vspace{0.3em} 	\\  \hline	
			
			$E_{n,i,s}$
			& 
			& $s-\frac{n-1}{2}N$ 
			& $\frac{n-1}{2}N - s $
			&  $0$
			& $2s-(n-1)N$
			\\ \hline
		\end{tabular}
	\end{tabular}
	\end{center}
\end{table} 
\end{landscape}



\section{Gerstenhaber brackets on $\hh^n(A)$}
\label{general brackets}

Let $X \in \hh^n(\alg)$ and $Y \in \hh^q(\alg)$ be elements in the $\ground$-basis for $\hh^{\bullet}(\alg)$ listed in~Sections~\ref{ss:0} and~\ref{ss:n}.
In this section, we describe methods to compute the bracket $[X,Y]$, and we provide these brackets explicitly when $X,Y$ are algebra generators of $\hh^{\bullet}(\alg)$, as listed in~Sections~\ref{m even ring} and~\ref{m odd ring}. 
\begin{method}\label{method0}
In Section~\ref{brackets HH1}, we showed that we can write
\begin{align*}
[\varphi_{1,0},X]&=aX,
&[\varphi_{1,0},Y]&=bY,\\
[\varphi_{1,0}+\psi_{1,0},X]&=a'X, 
&[\varphi_{1,0}+\psi_{1,0},Y]&=b'Y,
\end{align*}
for some scalars $a,a',b,b' \in \ground$. Gerstenhaber brackets satisfy the Jacobi identity, so we have:
$$
[\varphi_{1,0},[X,Y]]=[[\varphi_{1,0},X],Y]+[X,[\varphi_{1,0}, Y]]= (a+b)[X,Y],
$$
and similarly,
$$[\varphi_{1,0}+\psi_{1,0},[X,Y]]=(a'+b')[X,Y].$$
It follows that $[X,Y]$ is either zero or an eigenvector for $[\varphi_{1,0},-]$ and $[\varphi_{1,0}+\psi_{1,0},-]$, with eigenvalues $(a+b)$ and $(a'+b')$ respectively. If $[X,Y] \neq 0$, we can write $[X,Y] \in \hh^{n+q-1}(\alg)$ as a linear combination of basis elements with these exact eigenvalues. If there are no such basis elements, we know $[X,Y]=0$.\
\end{method}

\begin{method}\label{methodc}
	Once we know that $[X,Y]$ is a linear combination of certain eigenvectors, we can often use the Poisson identity
	$$[xy,z] = [x,z]y + (-1)^{|x|(|z|-1)} x[y,z]$$
	to compute the coefficients.
	For convenience, we list the cup products we use here, some of which already appear in~\cite[Theorem 4.8]{ST}. The starred $(*)$ cup products differ from the results given in~\cite[Theorem 4.8]{ST}. \\
	
\textbf{If $m$ is even, we have:} 
	\begin{align*}
	\varepsilon_i\varphi_{1,0}&=0&
	\varphi_{1,0}^2&=0&
	E_{1,j,s}&=f_j^s \varphi_{1,0}&
	\chi_{n,0}\varphi_{1,0}&=\varphi_{n+1, 0} \\
	\varepsilon_i\psi_{1,0}&=0&
	\psi_{1,0}^2&=0&
	E_{1,j,s}&=-f_j^s \psi_{1,0} (*) & 
	\chi_{n,0}\psi_{1,0}&=\psi_{n+1,0}\\
	\varepsilon_i\varphi_{m-1,-1}&=0&
	\varphi_{1,0}\varphi_{m-1,-1}&=0&
	E_{n+1,j,s}&=(-1)^{\frac{n}{2}j}E_{1,j,s}\chi_{n,0}&
	\chi_{n,0}\pi_{2,0} &=\pi_{n+2,0}\\
	\varepsilon_i\psi_{m-1,1}&=0&
	\psi_{1,0}\psi_{m-1,1}&=0&
	F_{n,j,s}&=(-1)^{\frac{n}{2}j}f_j^s \chi_{n,0}&
	\chi_{n,0}\chi_{2,0}&=\chi_{n+2,0}\\
	\varepsilon_if_j&=0&
	\varphi_{m-1, -1}\psi_{m-1,1}&=0&
	&&&\\
	f_i f_j&=\delta_{i,j}f_i^2&
	\varphi_{1,0}\psi_{1,0}&=mN\pi_{2,0}&
	\pi_{2,0}&=(-1)^i\varepsilon_i\chi_{2,0} (*)&
	\chi_{m,1}\varphi_{m-1,-1} &=0\\
	&&
	\varphi_{1,0}\psi_{m-1,1}&=m\pi_{m,1} &
	\pi_{m, 1}&=(-1)^i\varepsilon_i\chi_{m,1}&
	\chi_{m,-1}\psi_{m-1,1} &=0\\
	&&
	\psi_{1,0}\varphi_{m-1,-1}&=-m\pi_{m,-1} (*)&
	\pi_{m, -1}&=(-1)^i\varepsilon_i\chi_{m,-1}&
	& 
\end{align*}
for all $i,j\in\ZZ/m\ZZ$ and $1\leq s\leq N-1$, whenever both sides of the equation exist. \\
 
\textbf{If $m$ is odd, we have:} 
\begin{align*}\setlength{\columnsep}{5cm}
f_i f_j&=\delta_{i,j}f_i^2 
&  E_{n,j,s} &=f_j^s \varphi_{n, 0}  
&F_{n,j,s}&=f_j^s\chi_{n,0}\\
\varepsilon_i\chi_{4,0}&=0
&\chi_{n, 0} \varphi_{1,0} & = \varphi_{n+1, 0} 
&F_{n,j,s} \varphi_{1,0}  & = E_{n+1,j,s} \\ 
f_i\varphi&=-f_i\psi
& \chi_{n, 0} \psi_{1,0} & =  \psi_{n+1,0} 
&  F_{n,j,s} \psi_{1,0}  & = -E_{n+1,j,s} 
 \end{align*}
for all $i,j\in\ZZ/m\ZZ$ and $1\leq s\leq N-1$, whenever both sides of the equation exist.
\end{method}

In the following propositions, we describe the Gerstenhaber brackets among algebra generators of $\hh^{\bullet}(\alg)$, as listed in~Sections~\ref{m even ring} and~\ref{m odd ring}. Recall that the brackets with the generators $\varphi_{1,0}$ and $\psi_{1,0}$ were already computed in~Section~\ref{brackets HH1}, Table~\ref{tab:diagonaleven} and~Table~\ref{tab:diagonalodd}.

\begin{prop}
\label{prop:bracket-m-even}
Suppose $m \geq 3$ is even and let $i,j\in\ZZ/m\ZZ$. Then,
\begin{align*}
[\varepsilon_i,Y]
&=\begin{cases} 
\frac{(-1)^{i+1}}{m}(\varphi_{1,0}+\psi_{1,0}) & \text{ if } Y= \chi_{2,0} \\
(-1)^{i+1} \varphi_{m-1,-1} & \text{ if } Y = \chi_{m,-1} \\
(-1)^{i+1}\psi_{m-1,1} & \text{ if } Y = \chi_{m,1} \\
0 & \text{ if } Y\in\{\varepsilon_j, f_j, \varphi_{m-1,-1}, \psi_{m-1,1}\} ,
\end{cases} \\
[f_i,Y]&=0 \text{ if } Y\in\{\varepsilon_j, f_j, \chi_{2,0}, \varphi_{m-1,-1}, \psi_{m-1,1}, \chi_{m,-1}, \chi_{m,1}\},\\
[\chi_{2,0},Y]&=\begin{cases}
\chi_{m,-1} & \text{ if } Y= \varphi_{m-1,-1} \\
-\chi_{m,1} & \text{ if } Y= \psi_{m-1,1} \\
0 & \text{ if }  Y\in\{f_j, \chi_{2,0}, \chi_{m,-1}, \chi_{m,1}\} ,
\end{cases} \\
[\varphi_{m-1,-1},Y]&=0 \,\,\text{ if } Y\in\{\varepsilon_j, f_j,  \varphi_{m-1,-1}, \psi_{m-1,1}, \chi_{m,-1}, \chi_{m,1}\}, \\
[\psi_{m-1,1},Y]&=0 \,\,\text{ if } Y\in\{\varepsilon_j, f_j,  \varphi_{m-1,-1}, \psi_{m-1,1}, \chi_{m,-1}, \chi_{m,1}\},\\
[\chi_{m,-1},Y]&=0 \,\,\text{ if } Y\in\{f_j, \chi_{2,0}, \varphi_{m-1,-1}, \psi_{m-1,1}, \chi_{m,-1}, \chi_{m,1}\},\\
[\chi_{m,1},Y]&=0 \,\,\text{ if } Y\in\{f_j, \chi_{2,0}, \varphi_{m-1,-1}, \psi_{m-1,1}, \chi_{m,-1}, \chi_{m,1}\}.
\end{align*}
\end{prop}

\begin{proof}
	All brackets in the proposition can be computed by using~Methods~\ref{method0} and~\ref{methodc}. We provide a few examples of the computations.
	\begin{itemize}[leftmargin=*]
		\item $[\varepsilon_i, \varphi_{m-1,-1}]=0$:\\
	The eigenvalues for $\varepsilon_i$ and $\varphi_{m-1,-1}$ under $[\varphi_{1,0}+\psi_{1,0},-]$ are $0$ and $m$ respectively. Hence, 
		$[\varepsilon_i, \varphi_{m-1,-1}]\in \hh^{m-2}(\alg)$ is either zero or an eigenvector with eigenvalue~$m$. 
		Since $\hh^{m-2}(\alg)$  has basis elements 
		$\{\chi_{m-2,0}, \pi_{m-2,0}, F_{m-2,j,s}\}_{j,s}$, all with eigenvalue $0$ under $[\varphi_{1,0}+\psi_{1,0},-]$, we conclude that $[\varepsilon_i, \varphi_{m-1,-1}]=0$. \\
		
		\item $[\varepsilon_i,\chi_{2,0}]=\frac{(-1)^{i+1}}{m}(\varphi_{1,0}+\psi_{1,0})$:\\
		The eigenvalues for $\varepsilon_i$ and $\chi_{2,0}$ under $[\varphi_{1,0},-]$ are $N$ and $-N$ respectively. Hence, 
		$[\varepsilon_i,\chi_{2,0}]$ is either zero or an eigenvector with eigenvalue $0$. Since $\hh^{1}(\alg)$  has basis elements 
		$\{\varphi_{1,0}, \psi_{1,0}, E_{1,j,s}\}_{j,s}$, and $E_{1,j,s}$ has eigenvalue $s\neq 0$ under $[\varphi_{1,0},-]$, we know that $[\varepsilon_i,\chi_{2,0}]$ is a linear combination of $\varphi_{1,0}$ and $\psi_{1,0}$, say
		$[\varepsilon_i,\chi_{2,0}]=a\varphi_{1,0}+b\psi_{1,0}$.
		\\
		Using the cup products listed in~Method~\ref{methodc}, we can compute
		$$0=[\varepsilon_i\varphi_{1,0}, \chi_{2,0}]=
		[\varepsilon_i, \chi_{2,0}]\varphi_{1,0}+\varepsilon_i[\varphi_{1,0}, \chi_{2,0}]
		=-bmN\pi_{2,0}-\varepsilon_i N \chi_{2,0}
		=(-bm-(-1)^i)N\pi_{2,0},$$
		and similarly, $0=[\varepsilon_i\psi_{1,0}, \chi_{2,0}]
		=(am+(-1)^i)N\pi_{2,0}.$
		It follows that $a=b=\frac{(-1)^{i+1}}{m}$. \\
				
		\item $[\varepsilon_i,\chi_{m,1}]=(-1)^{i+1}\psi_{m-1,1}$:\\
		By Method~\ref{method0}, we find that $[\varepsilon_i,\chi_{m,1}]=a\psi_{m-1,1}$ for some $a\in\ground$.
		Moreover, we see that $a=(-1)^{i+1}$, using the cup products listed in~Method~\ref{methodc}:
		$$0=[\varepsilon_i\varphi_{1,0}, \chi_{m,1}]=
		[\varepsilon_i, \chi_{m,1}]\varphi_{1,0}+\varepsilon_i[\varphi_{1,0}, \chi_{m,1}]
		=-am\pi_{m,1}-\varepsilon_i m\chi_{m,1}
		=(-a-(-1)^i)m\pi_{m,1}.$$ 
		
		\item $[\chi_{2,0},\varphi_{m-1,-1}]=\chi_{m,-1}$:\\
		By Method~\ref{method0}, we find that $[\chi_{2,0},\varphi_{m-1,-1}]=a\chi_{m,-1}$ for some $a\in\ground$.
		Moreover, $a=1$ since
		\begin{align*}
		0&=[\chi_{2,0},\varepsilon_0\varphi_{m-1,-1}]=
		\varphi_{m-1,-1}[\varepsilon_0,\chi_{2,0}]
		+\varepsilon_0[\chi_{2,0},\varphi_{m-1,-1}] \\ 
		&=-\varphi_{m-1,-1}\frac{\varphi_{1,0}+\psi_{1,0}}{m}
		+a\varepsilon_0 \chi_{m,-1}
		=(-1+a)\pi_{m,-1}.
		\end{align*}
		
		\item $[\varphi_{m-1,-1},\chi_{m,1}]=0$:\\
		By Method~\ref{method0}, we find that $[\varphi_{m-1,-1},\chi_{m,1}]=a\chi_{2m-2,0}$ for some $a\in\ground$. We see $a=0$ since
		\begin{align*}
		0&=[\varphi_{1,0}\varphi_{m-1,-1},\chi_{m,1}]=
		[\varphi_{1,0},\chi_{m,1}]\varphi_{m-1,-1}
		-\varphi_{1,0}[\varphi_{m-1,-1},\chi_{m,1}] \\
		&=-m\chi_{m,1}\varphi_{m-1,-1}
		-a\varphi_{1,0}\chi_{2m-2,0}
		=-a\varphi_{2m-1,0}.
		\end{align*} 
		
		\item $[\varphi_{m-1,-1},\psi_{m-1,1}]=0$:\\
		By Method~\ref{method0}, we find that $[\varphi_{m-1,-1},\psi_{m-1,1}]=a \varphi_{2m-3,0} + b \psi_{2m-3,0}$ for some $a,b\in\ground$.
		Using that $\varphi_{2m-3,0}=\chi_{2m-4}\varphi_{1,0}$ and $\psi_{2m-3,0}=\chi_{2m-4}\psi_{1,0}$, we find
		\begin{align*}
		0&=[\varphi_{1,0}\varphi_{m-1,-1},\psi_{m-1,1}]=
		[\varphi_{1,0},\psi_{m-1,1}]\varphi_{m-1,-1}
		+\varphi_{1,0}[\varphi_{m-1,-1},\psi_{m-1,1}] \\
		&=(N-m)\psi_{m-1,1}\varphi_{m-1,-1}
		+b \varphi_{1,0}\chi_{2m-4}\psi_{1,0}
		=b mN \pi_{2,0}\chi_{2m-4}
		=b mN \pi_{2m-2,0}.
		\end{align*}
		Similarly, 
		$0=[\varphi_{m-1,-1},\psi_{1,0}\psi_{m-1,1}]
		=a mN \pi_{2m-2,0}$.
		It follows that $a=b=0$. \\
		
		\item $[\chi_{m,-1}, \chi_{m,1}]=0$: \\
		The eigenvalues for $\chi_{m,-1}$ and $\chi_{m,1}$ under $[\varphi_{1,0}+\psi_{1,0},-]$ are $m$ and $-m$ respectively. Hence, bracket 
		$[\chi_{m,-1}, \chi_{m,1}]\in \hh^{2m-1}(\alg)$ is either zero or an eigenvector with eigenvalue~$0$. The only basis elements in $\hh^{2m-1}(\alg)$ with eigenvalue $0$ under $[\varphi_{1,0}+\psi_{1,0},-]$ are  
		$\left\{\varphi_{2m-1,0}, \psi_{2m-1,0}, E_{2m-1,j,s}\right\}_{j,s}$.
		\\
		Now, the eigenvalues for $\chi_{m,-1}$ and $\chi_{m,1}$ under $[\varphi_{1,0},-]$ are $0$ and $-m$ respectively. Hence, 
		$[\chi_{m,-1}, \chi_{m,1}]$ is either zero or an eigenvector with eigenvalue $-m$. Since $\varphi_{2m-1,0}$ and $\psi_{2m-1,0}$ have eigenvalue $-(m-1)N\neq -m$ under $[\varphi_{1,0},-]$, and $E_{2m-1,j,s}$ has eigenvalue $s-(m-1)N\neq -m$ under $[\varphi_{1,0},-]$, we conclude that  $[\chi_{m,-1}, \chi_{m,1}]$ is zero. 
	\end{itemize}	
\end{proof}

\begin{prop}\label{prop:bracket-m-odd}
	Suppose $m \geq 3$ is odd. The brackets $[X,Y]$ equal zero for the generators
	 $$X,Y \in\left\{\varepsilon_i,f_i,F_{2,i,1}, \chi_{4,0}, \varphi_{m-1,-1}, \psi_{m-1,1}, \chi_{2m,2}, \chi_{2m,-2}\mid i\in \ZZ/m\ZZ\right\}.$$
\end{prop}

\begin{proof}
	All brackets in the proposition other than $[\varepsilon_i, F_{2,j,1}]$ for $i,j\in \ZZ/m\ZZ$ can be computed using~Method~\ref{method0}. To show $[\varepsilon_i, F_{2,j,1}]=0$, we first use Method~\ref{method0} to write $[\varepsilon_i, F_{2,j,1}]=\sum_{k\in \ZZ/m\ZZ}a_k E_{1,k,1}$ for some~$a_k\in \ground$.
	Noting that $\varepsilon_i\chi_{4,0}=0$, we find
	$$0=[\varepsilon_i\chi_{4,0}, F_{2,j,1}]
	=\chi_{4,0}[\varepsilon_i, F_{2,j,1}]
	+\varepsilon_i[\chi_{4,0}, F_{2,j,1}]
	=\sum_{k\in \ZZ/m\ZZ}a_k  \chi_{4,0} E_{1,k,1}
	=\sum_{k\in \ZZ/m\ZZ}a_k  E_{5,k,1}.$$
	It follows that $a_k=0$ for all $k\in \ZZ/m\ZZ$. Hence, $[\varepsilon_i, F_{2,j,1}]=0$.
\end{proof}



\section{Gerstenhaber brackets with $E_{1,j,s}$}
\label{E brackets}

In Section~\ref{cohomology}, we described a $\ground$-linear basis of $\hh^n(\alg)$ for every $n \geq 0$.
In particular, the basis for $\hh^1(\alg)$ was given by 
$$\left\{\varphi_{1,0}, \,\psi_{1,0}, \,E_{1,j,s} \mid j\in \ZZ/m\ZZ, \,1\leq s\leq N-1\right\}.$$
The brackets of $\varphi_{1,0}$ and $\psi_{1,0}$ with $\ground$-basis elements of $\hh^n(\alg)$ were computed in Section~\ref{brackets HH1}.
In this section, we compute the brackets of $E_{1,j,s}\in\hh^{1}(\alg)$ with all basis elements of $\hh^n(\alg)$. This allows us to give a complete description of the Lie structure of $\hh^1(\alg)$ and its action on $\hh^n(\alg)$ in Section~\ref{sec:lie-structure}.

\begin{prop}\label{prop:bracket-m-even-E} 
	Suppose $m \geq 3$ is even. Let $i,j\in\ZZ/m\ZZ$, $1\leq s,r\leq N-1$, and let indices $\alpha, \,\gamma,\,\beta$ be as in Sections~\ref{subsec:m even n even} and~\ref{subsec:m even n odd}. We write $E_{n,j,r}:=0$ and $F_{n,j,r}:=0$ for~$r\geq N$. Then,
\begin{align*} 
[E_{1,j,s}, \varepsilon_i]&=0\\
[E_{1,j,s}, f_i]&=\delta_{i,j} f_j^{s+1}\\
[E_{1,j,s},\chi_{n,\alpha}] &= 
\begin{cases}
-(-1)^{\frac{n}{2}j}\frac{n}{2}N F_{n,j,s} & \mbox{ if } \alpha = 0\\
0 & \mbox{ if } \alpha \neq 0, \end{cases}
\\
[E_{1,j,s},\pi_{n,\alpha}]&= 0,
\\
[E_{1,j,s},F_{n,i,r}]
&=\delta_{i,j} \left(r-\frac{n}{2}N\right)F_{n,j,s+r},
\\
[E_{1,j,s}, \varphi_{n,\gamma}]
&= \begin{cases}
-(-1)^{\frac{n-1}{2}j}(s+\frac{n-1}{2}N) E_{n,j,s} 
& \mbox{ if } \gamma = 0\\
0 & \mbox{ if } \gamma \neq 0, \end{cases} 
\\
[E_{1,j,s}, \psi_{n,\beta}]
&= \begin{cases}
(-1)^{\frac{n-1}{2}j}(s+\frac{n-1}{2}N) E_{n,j,s} 
& \mbox{ if } \beta = 0 \\
0 & \mbox{ if } \beta \neq 0 ,\end{cases}
\\
[E_{1,j,s}, E_{n,i,r}]
&= \delta_{i,j}\left(r-s-\frac{n-1}{2}N\right)E_{n,j,s+r}.
\end{align*}
\end{prop}

\begin{proof}
	The brackets of $E_{1,j,s}$ with $\chi_{n,\alpha}, \,\pi_{n,\alpha},\, \pi_{n,0},\,\varphi_{n,\gamma}, \,\psi_{n,\beta}$ for $\alpha,\,\gamma,\,\beta\neq 0$ 
	can be computed by using~Method~\ref{method0}.
	The remaining brackets in the proposition follow by
	referring to the cup products in~Method~\ref{methodc}. 
	Indeed, for $\alpha\in\hh^n(A)$, the Poisson identity tells us that
	\begin{align*}
	[E_{1,j,s}, \alpha]=[f_j^s \varphi_{1,0}, \alpha] = [f_j^s, \alpha]\varphi_{1,0} + f_j^s[\varphi_{1,0},\alpha],
	\end{align*} 
	which we use in the first two computations below:
\begin{itemize}[leftmargin=*]
		
	\item
	$[ E_{1,j,s},\chi_{n, 0}] =-(-1)^{\frac{n}{2}j}\frac{n}{2}N F_{n,j,s}:$\\
	Since $[f_j^s ,   \chi_{n, 0}]=[f_j^s ,\chi^{n/2}_{2, 0}]=0$
	and $f_j^s   \chi_{n, 0} = (-1)^{\frac n 2 j}  F_{n,j,s} $, we find
	$$[ E_{1,j,s},\chi_{n, 0}] =[ f_j^s , \chi_{n, 0}] \varphi_{1,0} +  f_j^s  [\varphi_{1,0},\chi_{n, 0}]
	=f_j^s  [\varphi_{1,0},\chi_{n, 0}]
	=- \frac n 2 N f_j^s   \chi_{n, 0}
	=  -(-1)^{\frac n 2 j} \frac n 2 N  F_{n,j,s}.$$ 

	\item 
	$[ E_{1,j,s},F_{n,i,r}] = \delta_{i,j} \left(r-\frac{n}{2}N\right)F_{n,j,s+r}$:\\
	Note that $[ f_j^s , F_{n,i,r}] = [ f_j^s , (-1)^{\frac n 2 i} f_i^r \chi_{n,0}]  = 0$ because $ [ f_j^s , f_i^r ] =0$ and $ [ f_j^s ,  \chi_{n,0}] =0$.  Hence,
	$$[ E_{1,j,s},F_{n,i,r}] 
	=  [ f_j^s , F_{n,i,r}] \varphi_{1,0} +  f_j^s[\varphi_{1,0},F_{n,i,r}]
	= \left(r-\frac{n}{2}N\right) f_j^sF_{n,i,r}
	=\delta_{i,j} \left(r-\frac{n}{2}N\right)F_{n,j,s+r}.$$

	\item 
	$[ E_{1,j,s},\varphi_{n, 0}] =-(-1)^{\frac{n-1}{2}j}\left(s+\frac{n-1}{2}N\right) E_{n,j,s}$:\\
	Since $\varphi_{n,0}=\chi_{n-1,0}\varphi_{1,0}$, we see
	\begin{align*}[ E_{1,j,s},\varphi_{n, 0}]
	&=[ E_{1,j,s},\chi_{n-1,0}]\varphi_{1, 0}
	-\chi_{n-1,0}[ \varphi_{1, 0},E_{1,j,s}] \\
	&=-(-1)^{\frac{n-1}{2}j}\left( \frac{n-1}{2}\right) N F_{n-1,j,s} \varphi_{1, 0}
	-s\chi_{n-1,0} E_{1,j,s}\\
	&=-(-1)^{\frac{n-1}{2}j}\left( \frac{n-1}{2}\right) NE_{n,j,s}
	-(-1)^{\frac{n-1}{2}j}sE_{n,j,s}.
	\end{align*}

	\item 
	$[ E_{1,j,s},\psi_{n, 0}] =   (-1)^{\frac{n-1}{2} j }\left(s+\frac{n-1}{2}N \right)  E_{n,j,s}$:\\
	Since $\psi_{n,0}=\chi_{n-1,0}\psi_{1,0}$, we see
	\begin{align*}[ E_{1,j,s},\psi_{n, 0}]
	&=[ E_{1,j,s},\chi_{n-1,0}]\psi_{1, 0}
	-\chi_{n-1,0}[ \psi_{1, 0},E_{1,j,s}] \\
	&=-(-1)^{\frac{n-1}{2}j}\left( \frac{n-1}{2}\right) N F_{n-1,j,s} \psi_{1, 0}
	+s\chi_{n-1,0} E_{1,j,s}\\
	&=(-1)^{\frac{n-1}{2}j}\left( \frac{n-1}{2}\right) NE_{n,j,s}
	+(-1)^{\frac{n-1}{2}j}sE_{n,j,s}.
	\end{align*}
 
	\item 
	$[ E_{1,j,s},  E_{n,i,r}] = \delta_{i,j}\left(r-s-\frac{n-1}{2}N\right)E_{n,j,s+r}$:\\
	Since $E_{n,i,r}=F_{n-1,i,r}\varphi_{1,0}$, we see
	\begin{align*}
	[ E_{1,j,s},  E_{n,i,r}] 
	&=  [ E_{1,j,s},  F_{n-1,i,r}] \varphi_{1,0} 
	-  F_{n-1,i,r}  [\varphi_{1,0},  E_{1,j,s}] \\
	&=\delta_{i,j} \left(r-\frac{n-1}{2}N\right)F_{n-1,j,s+r}\varphi_{1,0} 
	-sF_{n-1,i,r}E_{1,j,s}\\ 
	&=\delta_{i,j} \left(r-\frac{n-1}{2}N\right) E_{n,j,s+r}   -\delta_{i,j}  s  E_{n,j ,s+r}
	=   \delta_{i,j}   \left(r-s -\frac {n-1}{2} N \right)  E_{n,j,s+r} .
	\end{align*}
	\end{itemize}
This proves the proposition.
\end{proof}

\begin{prop}\label{prop:bracket-m-odd-E} 
	Suppose $m \geq 3$ is odd. Let $i,j\in\ZZ/m\ZZ$, $1\leq s,\,r\leq N-1$ and let indices $\delta, \,\sigma,\,\tau$ be as in Sections~\ref{subsec:m odd n even} and~\ref{subsec:m odd n odd}.  We write $E_{n,j,r}:=0$ and $F_{n,j,r}:=0$ for~$r\geq N$. Then,
	\begin{align*} 
	[E_{1,j,s}, \varepsilon_i]&=0\\
	[E_{1,j,s}, f_i]&=\delta_{i,j} f_j^{s+1}\\
	[E_{1,j,s},\chi_{n,\delta}]
	&= \begin{cases}
	-\frac{n}{2}N F_{n,j,s} &\mbox{ if } \delta = 0\\
	0&\mbox{ if } \delta \neq 0, \end{cases}
	\\
	[E_{1,j,s},\pi_{n,\delta}]&= 0,
	\\
	[E_{1,j,s},F_{n,i,r}]
	&=\delta_{i,j} \left(r-\frac{n}{2}N\right)F_{n,j,s+r},
	\\
	[E_{1,j,s}, \varphi_{n,\sigma}]
	&= \begin{cases}
	-\left( s+\frac{n-1}{2}N\right)  E_{n,j,s} &\mbox{ if } \sigma = 0 \\
	0& \mbox{ if } \sigma \neq 0 , \end{cases}
	\\
	[E_{1,j,s}, \psi_{n,\tau}]
	&= \begin{cases}
	\left( s+\frac{n-1}{2}N\right)  E_{n,j,s}&\mbox{ if } \tau = 0 \\
	0& \mbox{ if } \tau \neq 0 , \end{cases}
	\\
	[E_{1,j,s}, E_{n,i,r}]
	&=  \delta_{i,j}\left(r-s-\frac{n-1}{2}N\right)E_{n,j,s+r}.
	\end{align*}
\end{prop}
\begin{proof}
	The brackets of $E_{1,j,s}$ with $\chi_{n,\delta}, \,\pi_{n,\delta},\, \pi_{n,0}, \,\varphi_{n,\sigma}, \,\psi_{n,\tau}$ for $\delta,\,\sigma,\,\tau\neq 0$ can be computed using~Method~\ref{method0}.
	The remaining brackets in the proposition follow by
	referring to the cup products in~Method~\ref{methodc}. Again, for $\alpha\in\hh^n(A)$, the Poisson identity tells us that
	\begin{align*}
	[E_{1,j,s}, \alpha]=[f_j^s \varphi_{1,0}, \alpha] = [f_j^s, \alpha]\varphi_{1,0} + f_j^s[\varphi_{1,0},\alpha].
	\end{align*} 
	Moreover, note that $[f_j,\alpha]=0$ for all $\alpha\in \hh^n(\alg)$ with $n$ even. This follows because $[f_j,-]$ is zero on all generators of $\hh^\bullet(\alg)$ in even degree, and since $$[\varphi_{1,0}\psi_{1,0},f_j]=[\varphi_{1,0},f_j]\psi_{1,0}-\varphi_{1,0}[\psi_{1,0},f_j]=f_j\psi_{1,0}+f_j\varphi_{1,0}=0.$$
	The rest of the proof follows as in the proof of~Proposition~\ref{prop:bracket-m-even-E}.
	
\end{proof}

We summarize the brackets $[E_{1,j,s},\hh^n(\alg)]$ in Table~\ref{tab:Ebrackets}. As before, we refer the reader to Section~\ref{ss:n} for explicit basis elements of $\hh^n(\alg)$ for $m$ even and $m$ odd, and we let $i \in \mathbb{Z}/m\mathbb{Z}$.

\FloatBarrier
{\small 
\begin{table}[!htp] 
		\caption{Brackets $[E_{1,j,s},\hh^n(\alg)]$} \label{tab:Ebrackets}
                 \vspace{.1cm}
	\begin{tabular}{c@{}l}
			$\begin{array}{r}
			\vspace{-0.5ex}\\
			\MyLBrace{7.5ex}{$n=0$} 
			\vspace{1.5ex}\\ 
			\vspace{1.5ex}
			\MyLBrace{12.5ex}{$n$ even} \\
			\MyLBrace{13ex}{$n$ odd} 
			\end{array}$
			&
		\begin{tabular}{| p{0.65cm} p{3.55cm} | p{4.55cm} | p{4.55cm} |}
			
			\hline
			$\mathbf{\hh^n}$ & 
			& $[E_{1,j,s}, - ]$, $m$   even
			& $[E_{1,j,s}, - ]$, $m$ odd  \\
			 \hline   \hline
			 
			$1$ &
			& $0$
			& $0$ 
			\\ \hline
			
			$\varepsilon_i$ 
			& 
			&  $0$
			& $0$
			\\ \hline
			
			$f_i^r$  
			&   $\begin{cases}
			 \mbox{for } 1 \leq r \leq  N-s  \\
					 \mbox{otherwise } \end{cases}	$
			&  $\begin{cases}
			\delta_{i,j}r f_i^{r+s} \\
					0 \end{cases}	$
			&  $\begin{cases}
			\delta_{i,j} rf_i^{r+s} \\
					0 \end{cases}	$
						\\ \hline   \hline \hline
			
			$\chi_{n,\delta}$  
			& $\begin{cases}
			 \mbox{for }\delta = 0  \\
					 \mbox{otherwise } \end{cases}	$	
			& $\begin{cases}
			 -(-1)^{\frac n 2 j} \frac n 2 N  F_{n,j,s} \\
			 0 \end{cases}$
			& $\begin{cases}
					 -  \frac n 2 N  F_{n,j,s}   \\
										0 \end{cases}$   
			\vspace{0.5em}\\  \hline

			$\pi_{n,\delta}$  
			& 		
			& 0
			& 0 
			\vspace{0.3em}\\  \hline
			
			$F_{n,i,r}$ 
			& $\begin{cases}
			 \mbox{for } 1 \leq r \leq  N-s-1  \\
					 \mbox{otherwise } \end{cases}	$	
			& $\begin{cases} \delta_{i,j} ( r- \frac n 2) N   F_{n,j,r+s}  \\
			 0 \end{cases}$
			&  
			 $\begin{cases} \delta_{i,j} ( r- \frac n 2) N   F_{n,j,r+s} \\
			  0 \end{cases}$	
			 \\  \hline
			
			$\varphi_{n,\sigma}$ 
			& 
			& --------------
			& $0$
			\\  \hline
			
			$\psi_{n,\tau}$
			& 
			& --------------
			& 0			
			\\  \hline
			\hline \hline
			
			$\pi_{n,\delta}$
			& 
			& --------------
			& 0
			\vspace{0.3em} \\  \hline
			
			$\varphi_{n,\sigma}$ 
			& $\begin{cases}
						\mbox{for } \sigma =0 \\
						\mbox{otherwise } \end{cases} $
			& $\begin{cases}
						- (-1)^{\frac{n-1}{2} j } (s +\frac {n-1}{ 2} N )  E_{n,j,s} \\
						0  \end{cases}$ 
			& $\begin{cases}
						- (s +\frac {n-1}{ 2} N )  E_{n,j,s}  \\
						0  \end{cases}$ 
			\vspace{0.3em} \\  \hline
			
			$\psi_{n,\tau}$ 
			& $\begin{cases}
						\mbox{for } \tau =0\\ 
						\mbox{otherwise } \end{cases} $
			& $\begin{cases}
						 (-1)^{\frac{n-1}{2} j } (s  + \frac {n-1}{ 2} N )  E_{n,j,s}  \\
						0  \end{cases}$ 
			& $\begin{cases}
						 (s +\frac {n-1}{ 2} N )  E_{n,j,s}  \\
						0  \end{cases}$ 		
			\vspace{0.3em} 	\\  \hline	
			
			$E_{n,i,r}$
			& $\begin{cases}
			 \mbox{for } 1 \leq r  \leq N-s-1  \\
					 \mbox{otherwise } \end{cases}	$
			& $\begin{cases} \delta_{i,j}   (-s + r -\frac {n-1}{ 2} N )  E_{n,j,r+s} \\
			 0 \end{cases}$
		        & $\begin{cases} \delta_{i,j}   (-s + r -\frac {n-1}{ 2} N )  E_{n,j,r+s} \\
				 0 \end{cases}$		
				\\ \hline
		\end{tabular}. 
	\end{tabular}	
\end{table} }
\FloatBarrier



\section{The Lie algebra $\hh^1(\alg)$ and the Lie modules $\hh^n(\alg)$}
\label{sec:lie-structure}

\numberwithin{equation}{subsection}

In this section, we describe the Lie structure of the first Hochschild cohomology space 
$\hh^1(\alg)$ and the Lie module structure of $\hh^n(\alg)$ over $\hh^1(\alg)$.

\subsection{The Lie structure of $\hh^1(\alg)$}
Recall that $\hh^1(\alg)$ has a $\ground$-basis given by
$$\left\{\varphi_{1,0}, \,\psi_{1,0}, \,E_{1,j,s}
\mid j\in \ZZ/m\ZZ, \,1\leq s\leq N-1\right\}.$$

We write
$$C:=\frac12(\varphi_{1,0}+\psi_{1,0}) \qquad \text{ and } \qquad E_0:=\frac12(\varphi_{1,0}-\psi_{1,0}).$$ 

The brackets among these elements of $\hh^1(\alg)$ are given by
\begin{align*}
[C,E_0]  & =0, \\
[C,E_{1,j,s}] & =0, \\
[E_0,E_{1,j,s}] & =sE_{1,j,s},\\
[E_{1,j,s},E_{1,i,r}] & =
\begin{cases}
\delta_{i,j}(r-s)E_{1,j,s+r} & \hbox{if } s+r \le N-1, \\
0 & \hbox{ otherwise.}
\end{cases}
\end{align*}

\begin{prop}
The center of the Lie algebra $\hh^1(\alg)$ is given by the $\ground$-span of $C$. In particular, the center is one-dimensional.
\end{prop}
\begin{proof}
Direct computations of the brackets $[C,-]$ on a basis of $\hh^1(\alg)$ imply that $C$ belongs to the center. In order to prove that $C$ generates the center as a vector space, we consider an element in the center and write it as a linear combination of $C$, $E_0$ and the $E_{1,j,s}$'s. Computing the bracket with $E_0$, we see that the coefficient of each $E_{1,j,s}$ must be zero. Computing the brackets with each $E_{1,j,s}$, we see that the coefficient of $E_0$ must also be zero.
\end{proof}

\begin{rem} 
From the bracket computations, we see that $\hh^1(\alg)$ is a solvable Lie algebra, since its derived series stops after at most $\lfloor N/2\rfloor$ steps.
\end{rem}

We note that $\hh^1(\alg)$ has a Lie subalgebra with $\ground$-basis 
$\{E_0, \, E_{1,j,s} \mid 1\le s\le N-1\}$ for each fixed $j \in \ZZ/m\ZZ$.
We will show that each of these Lie subalgebras is isomorphic to a subquotient of the Virasoro algebra. 

\begin{definition}
Recall that the \textbf{Virasoro algebra} $\vir$ is the unique central extension of the Witt algebra, with generators $L_s$ and $c$, where $c$ is central and $s \in \ZZ$, and brackets given by 
$$[L_s,L_r]=(r-s)L_{s+r}+\delta_{s+r,0}\frac{s^3-s}{12}c.$$
Denote the Lie subalgebra generated by all $L_s$ for $s \ge0$ by $\vir_+$. For $q \ge0$, let $\vir_{> q}$ be the Lie ideal of $\vir_+$ generated by all $L_s$ with $s > q$, and write $\mathfrak{a}_q$ for the subquotient $\mathfrak{a}_q:=\vir_+/\vir_{>q}$. We denote the residue classes in $\mathfrak{a}_q$ of the generators of $\vir$ again by $L_s$.
\end{definition}

This subquotient $\mathfrak{a}_q$ of the Virasoro algebra was studied in \cite{LLZ, Mazorchuk-Zhao}. It inherits a grading from the standard $\ZZ$-grading of the Virasoro algebra, in which $L_s$ has degree equal to $s \in\ZZ$ and $c$ has degree~$0$. In order to avoid confusion with the cohomological degree, we refer to this grading as the Virasoro grading.

\begin{rem}
A classification of the irreducible modules over $\mathfrak{a}_1$ was given in \cite{Block, Mazorchuk}, and a classification of the irreducible modules over $\mathfrak{a}_2$ was obtained in \cite{Mazorchuk-Zhao}. For $\mathfrak{a}_r$ with $r\ge 3$, the classification problem is open.
We note that the Lie algebras of dimensions $2$ and $3$ are completely classified up to isomorphism, see for example 
\cite{agaoka}. In this classification, $\mathfrak{a}_1$ is the unique -- up to isomorphism -- non-abelian Lie algebra of dimension $2$, 
denoted by $\mathfrak{aff}(2)$. In dimension~$3$, six types of algebras appear, one of which is an infinite family $\mathfrak{\tau}_{\alpha}$ depending 
on a complex nonzero parameter $\alpha$, and $\mathfrak{a}_2$ is isomorphic to the Lie algebra $\mathfrak{\tau}_2$, see~\cite{agaoka}.
\end{rem}

By mapping $L_0$ to $E_0$ and $L_s$ to $E_{1,j,s}$ for $1\le s\le N-1$, we obtain an isomorphism between the subquotient $\mathfrak{a}_{N-1}$ of the Virasoro algebra and a Lie subalgebra of $\hh^1(\alg)$:

\begin{prop}
	For each $j\in \ZZ/m\ZZ$, there is an isomorphism of Lie algebras
	$$\mathfrak{a}_{N-1} \cong \langle E_0,E_{1,j,s}\ |\ 1\le s\le N-1\rangle.$$
\end{prop}
Here and in what follows we use the notation $\langle \ldots \rangle$ to denote the $\ground$-linear span of the given elements. We conclude that $\hh^1(\alg)$ contains $m$ copies of the Lie algebra $\mathfrak{a}_{N-1}$, which share Virasoro degree $0$ and commute otherwise.
We have the following result:
\begin{thm}
\label{th:vir}
Let $N \geq 1$ and $m \geq 3$. There is an embedding of Lie algebras
\begin{align*}
\hh^1(\alg)\ &\hookrightarrow \ \langle c \rangle \oplus \left(\bigoplus_{j=0}^{m-1} \mathfrak{a}_{N-1}(j) \right) \\
E_{1,j,s}\ &\mapsto\ L_s(j)\\
E_0\ &\mapsto\ \sum_{j=0}^{m-1} L_0(j)\\
C\ &\mapsto\ c,
\end{align*}
where $L_0(j), L_s(j)$ are in $\mathfrak{a}_{N-1}(j) = \mathfrak{a}_{N-1}$ for all $j\in \ZZ/m\ZZ$, and $c$ is a central element.

More precisely, $\hh^1(\alg)$ is the pullback of the following diagram of Lie algebras:
$$(\langle c \rangle \oplus\mathfrak{a}_{N-1}(j)	\twoheadrightarrow \langle c, L_0 \rangle)_{0\leq j \leq m-1},$$
where $\langle c, L_0 \rangle$ is a commutative Lie algebra.
\end{thm}


\subsection{Decomposition of $\hh^n(\alg)$ as an $\hh^1(\alg)$-module when $n\geq 1$}

In~Proposition~\ref{diagonal}, we showed that the elements $C, E_0$ in $\hh^1(\alg)$ act diagonally on a $\ground$-basis of $\hh^n(\alg)$, see Table~\ref{tab:diagonaleven} and~Table~\ref{tab:diagonalodd}. The action of $E_{1,j,s}$ on this basis is given in Table~\ref{tab:Ebrackets}. We thus obtain the following decomposition of $\hh^n(\alg)$ into indecomposable summands as a module over $\hh^1(\alg)$.

\begin{thm}
Let $m \geq 3$ be even and $n\geq 1$. Write $n=pm+t$ with $p\geq 0$ and $0\le t\le m-1$ as before. We can decompose $\hh^n(\alg)$ into indecomposable summands over $\hh^1(\alg)$ as follows: 
If $n$ is even,
$$\hh^n(A) = \displaystyle 
	\bigoplus_{\substack{-p\le \alpha\le p\\ \alpha\neq 0}} 
	\langle \chi_{n,\alpha}\rangle 
	\oplus 
	\bigoplus_{-p\le \alpha\le p} 
	\langle \pi_{n,\alpha}\rangle  
	\ \oplus\ 
	\langle \chi_{n,0}, F_{n,j,s} \mid j, s \rangle,$$
where $j\in \ZZ/m\ZZ$ and $1\leq s\leq N-1$.\\
If $n$ is odd, 
	$$
	\hh^n(A) = 
	\begin{cases} 
	\displaystyle 
	\bigoplus_{\substack{-p\le \beta\le p\\ \beta \neq 0}}
	\langle \psi_{n,\beta}\rangle
	\oplus
	\bigoplus_{\substack{-p\le \gamma\le p\\ \gamma\neq 0}} 
	\langle \varphi_{n,\gamma}\rangle 
	\ \oplus\ 
	\langle \varphi_{n,0} + \psi_{n,0}\rangle 
	\ \oplus\ 
	\langle \varphi_{n,0}, E_{n,j,s} \mid j, s \rangle 
	& \text{ if } t \neq m-1,\\
	\\
	\displaystyle 
	\!\bigoplus_{\substack{-p\le \beta \le p+1\\ \beta\neq 0}}
	\!\!\!\langle \psi_{n,\beta}\rangle 
	\oplus 
	\!\!\!\bigoplus_{\substack{-p-1\le \gamma\le p\\ \gamma\neq 0}} \!\langle \varphi_{n,\gamma}\rangle 
	\ \oplus\ 
	\langle \varphi_{n,0} + \psi_{n,0} \rangle
	\ \oplus\ 
	\langle \varphi_{n,0}, E_{n,j,s} \mid j, s \rangle 
	& \text{ if } t=m-1,
	\end{cases}
	$$
where $j\in \ZZ/m\ZZ$ and $1\leq s\leq N-1$. 
\end{thm}

When $m$ is odd, Remark~\ref{rem:indexchi} and Remark~\ref{rem:indexphi} show $\varphi_{n,0}$ and $\psi_{n,0}$ exist if and only if $n\equiv 1 \text{ (mod 4)}$, and $\chi_{n,0}$ exists if and only if $n\equiv 0  \text{ (mod 4)}$. We thus have to consider more cases in the description of $\hh^n(\alg)$ as an $\hh^1(\alg)$-module, but the resulting decompositions are similar in nature to the case when $m$ is even.

\begin{thm}
Let $m \geq 3$ be odd and $n\geq 1$. Write $n=pm+t$ with $p\geq 0$ and $0\le t\le m-1$ as before. We can decompose $\hh^n(\alg)$ into indecomposable summands over $\hh^1(\alg)$ as follows: 
\begin{enumerate}
\item  If $n$ is even and $n\equiv 0 \text{ (mod 4)}$,
$$
\hh^n(\alg)\ =\ \begin{cases} 
\displaystyle \bigoplus_{\substack{\delta\in I_\chi^n\\\delta\neq 0}} 
\langle \chi_{n,\delta}\rangle 
\oplus
\bigoplus_{\delta\in I_\pi^n} 
\langle \pi_{n,\delta}\rangle 
\oplus
\langle \chi_{n,0}, F_{n,j,s}\mid j,s\rangle 
& \text{ if }t\neq m-1 ,\\
\ \\
\displaystyle \bigoplus_{\substack{\delta\in I_\chi^n\\\delta\neq 0}} 
\langle \chi_{n,\delta}\rangle 
\oplus
\bigoplus_{\delta\in I_\pi^n} 
\langle \pi_{n,\delta}\rangle 
\oplus 
\langle \chi_{n,0},  F_{n,j,s} \mid j,s\rangle\oplus
\langle \varphi_{n,-(p+1)}\rangle  
\oplus
\langle \psi_{n,p+1}\rangle  
  & \text{ if }t=m-1;
\end{cases}
$$
\item  If $n$ is even and $n \not \equiv 0 \text{ (mod 4)}$,
$$
\hh^n(\alg)\ =\ \begin{cases} 
\displaystyle \bigoplus_{\substack{\delta\in I_\chi^n\\\delta\neq 0}} 
\langle \chi_{n,\delta}\rangle 
\oplus
\bigoplus_{\delta\in I_\pi^n} 
\langle \pi_{n,\delta}\rangle 
\oplus \bigoplus_{j\in \ZZ/m\ZZ} 
\langle F_{n,j,s}\mid s\rangle 
& \text{ if }t\neq m-1 ,\\
\ \\
\displaystyle \bigoplus_{\substack{\delta\in I_\chi^n\\\delta\neq 0}} 
\langle \chi_{n,\delta}\rangle 
\oplus
\bigoplus_{\delta\in I_\pi^n} 
\langle \pi_{n,\delta}\rangle 
\oplus \bigoplus_{j\in \ZZ/m\ZZ} 
\langle  F_{n,j,s}\mid s\rangle\oplus
\langle \varphi_{n,-(p+1)}\rangle  
\oplus
\langle \psi_{n,p+1}\rangle  
  & \text{ if }t=m-1;
\end{cases}
$$
\item  If $n$ is odd and $n  \equiv 1 \text{ (mod 4)}$,
$$
\hh^n(\alg)\ =\ \begin{cases} 
\displaystyle 
\bigoplus_{\substack{\tau\in I_\psi^n\\\tau\neq 0}} 
\langle \psi_{n,\tau}\rangle 
\oplus
\bigoplus_{\substack{\sigma\in I_\varphi^n\\\sigma\neq 0}} 
\langle \varphi_{n,\sigma}\rangle 
\oplus
\langle \varphi_{n,0} + \psi_{n,0}\rangle  
\oplus
\langle \varphi_{n,0}, E_{n,j,s}\mid j,s\rangle 
& \text{ if }t\neq 0, \\
\ \\
\displaystyle 
\bigoplus_{\substack{\tau\in I_\psi^n\\\tau\neq 0}} 
\langle \psi_{n,\tau}\rangle 
\oplus
\bigoplus_{\substack{\sigma\in I_\varphi^n\\\sigma\neq 0}} 
\langle \varphi_{n,\sigma}\rangle 
\oplus
\langle \varphi_{n,0} + \psi_{n,0}\rangle  
\oplus
\langle \varphi_{n,0}, E_{n,j,s}\mid j,s\rangle 
\oplus 
\bigoplus_{\delta=\pm p} 
\langle \pi_{n,\delta}\rangle 
& \text{ if }t=0;
\end{cases}
$$
\item  If $n$ is odd and $n \not  \equiv 1 \text{ (mod 4)}$,
$$
\hh^n(\alg)\ =\ \begin{cases} 
\displaystyle 
\bigoplus_{\substack{\tau\in I_\psi^n\\\tau\neq 0}} 
\langle \psi_{n,\tau}\rangle 
\oplus
\bigoplus_{\substack{\sigma\in I_\varphi^n\\\sigma\neq 0}} 
\langle \varphi_{n,\sigma}\rangle 
\oplus \bigoplus_{j\in \ZZ/m\ZZ} 
\langle E_{n,j,s}\mid s\rangle 
& \text{ if }t\neq 0, \\
\ \\
\displaystyle 
\bigoplus_{\substack{\tau\in I_\psi^n\\\tau\neq 0}} 
\langle \psi_{n,\tau}\rangle 
\oplus
\bigoplus_{\substack{\sigma\in I_\varphi^n\\\sigma\neq 0}} 
\langle \varphi_{n,\sigma}\rangle 
\oplus \bigoplus_{j\in \ZZ/m\ZZ} 
\langle E_{n,j,s}\mid s\rangle 
\oplus 
\bigoplus_{\delta=\pm p} 
\langle \pi_{n,\delta}\rangle 
& \text{ if }t=0,
\end{cases}
$$
\end{enumerate}
where $j\in \ZZ/m\ZZ$, $1\leq s\leq N-1$, and $I_\chi^n$, $I_\pi^n, I_\varphi^n$, $I_\psi^n$ are given by the respective basis index sets of $\hh^n(\alg)$ as described  in~Section~\ref{subsec:m odd n even} and in~Section~\ref{subsec:m odd n odd}.
\end{thm}

\begin{rem}
In the above decompositions, almost all the one-dimensional summands are nontrivial simple modules. The only exception is $\langle \pi_{2,0}\rangle$, on which $\hh^1(\alg)$ acts trivially for both $m$ even and $m$ odd cases. 
\end{rem}

\begin{rem}
		Table~\ref{tab:diagonaleven} and~Table~\ref{tab:diagonalodd} allow us to describe the central characters of the indecomposable summands of $\hh^n(\alg)$ for both $m$ even and $m$ odd. Indeed, we see that the central element $C$ acts by $-\frac{\alpha m}{2}$ on the indecomposable summands $\langle \chi_{n,\alpha}\rangle$, $\langle \pi_{n,\alpha}\rangle$ $\langle \psi_{n,\alpha}\rangle$ and $\langle \varphi_{n,\alpha}\rangle$ whenever $\alpha\neq 0$. Furthermore, $C$ acts trivially on the remaining indecomposable summands. 
\end{rem}

\begin{rem}
Let $m$ be even or $m$ odd and $n\equiv 1 \text{ (mod 4)}$.
The $\hh^1(\alg)$-module $\langle \varphi_{n,0}, E_{n,j,s}\mid j,s\rangle$ has dimension $(1+ m(N-1))$ and is a weight module with respect to the action of $E_0$. It is generated as an $\hh^1(\alg)$-module by $\varphi_{n,0}$, on which $E_0$ acts by $\frac{-(n-1)N}{2}$. For every $1\le s\le N-1$, the weight space of weight $s-\frac{(n-1)}{2}N$ is given by the $m$-dimensional subspace $\langle E_{n,j,s}\mid j\rangle$. Note that for every $j\in\ZZ/m\ZZ$, the action of $E_{1,j,1}$ increases the weight by one. 
We note that any subspace of the form 
$$\langle E_{n,j,s} \mid j\in \ZZ/m\ZZ,\; s_j\le s\le N-1 \rangle,$$
where  $1\le s_j\le N-1$ for every $j\in \ZZ/m\ZZ$, is an $\hh^1(\alg)$-submodule of $\langle \varphi_{n,0}, E_{n,j,s}\mid j,s\rangle$. Hence, the simple subquotients of $\langle \varphi_{n,0}, E_{n,j,s}\mid j,s\rangle$ are all one-dimensional.
A similar description applies to the $\hh^1(\alg)$-module $\langle \chi_{n,0}, F_{n,j,s}\mid j,s\rangle$, which is generated by the element $\chi_{n,0}$.
\end{rem}

\subsection{Decomposition of $\hh^0(\alg)$ as an $\hh^1(\alg)$-module}

The $\ground$-module $\hh^0(\alg)$ is the same for $m$ even or $m$ odd. In both cases, it has a basis given by 
$$\{1,\varepsilon_i,f_i^s\ |\ i\in \ZZ/m\ZZ,\ 1\le s\le N-1\}.$$ 
However, we treat the two cases separately as the decompositions as an $\hh^1(\alg)$-module are different. In~Proposition~\ref{diagonal}, we showed that the elements $C, E_0$ in $\hh^1(\alg)$ act diagonally on $\hh^0(\alg)$ with respect to the above basis. The action of $E_{1,j,s}$ on this basis is given in Table~\ref{tab:Ebrackets}. 
In particular, we know that for any $i,j\in \ZZ/m\ZZ$ and $1\le r, s\le N-1$, 
$$[E_{1,j,r},f_i^s]\ =\ \delta_{i,j} sf_i^{r+s},$$
which is zero if $r+s> N$. For $r+s=N$ we can rewrite $f_i^N=\varepsilon_i+\varepsilon_{i+1}$ in terms of the above basis by~\cite[Theorem 4.8]{ST}.

\begin{thm}
	For $m\ge 3$ even, $\hh^0(\alg)$ has the following decomposition into indecomposable $\hh^1(\alg)$-modules:
	$$\hh^0(\alg)\ =\ \langle 1\rangle\ \oplus\ \langle \varepsilon_0\rangle\ \oplus \ \langle f_i^s, \varepsilon_i+\varepsilon_{i+1} \mid i\in\ZZ/m\ZZ, 1\le s\le N-1 \rangle.$$
\end{thm}

\begin{proof}
	We first note that when $m$ is even,
	$$0= \sum_{j\in\ZZ/m\ZZ} (-1)^{[j]}(\varepsilon_j+\varepsilon_{j+1}),$$ 
	where $[j]\in\{0,\ldots,m-1\}$ denotes the unique representative of $j\in\ZZ/m\ZZ$.
	Hence, the elements $\varepsilon_j+\varepsilon_{j+1}$ with $j\in\ZZ/m\ZZ$ are linearly dependent.
	
	Next, we show that for any direct sum decomposition $V\oplus W=\langle f_i^s, \varepsilon_i+\varepsilon_{i+1} \mid i,s\rangle$, one of the submodules $V, W$ is zero. 
	Note that $[E_{1,j,N-1},-]$ maps $V\oplus W$ onto $\langle \varepsilon_j+\varepsilon_{j+1}\rangle$. Since
	\begin{align*}
	[E_{1,j,N-1},V]\subset V\cap\langle \varepsilon_j+\varepsilon_{j+1}\rangle \  \ \text{and}\ \ 
	[E_{1,j,N-1},W]\subset\ W\cap\langle \varepsilon_j+\varepsilon_{j+1}\rangle,
	\end{align*}
	$\varepsilon_j+\varepsilon_{j+1}$ is contained in $V$ or $W$.
	Since the set $\{\varepsilon_j+\varepsilon_{j+1}\ |\ j\in\ZZ/m\ZZ\}$ is linearly dependent, all $\varepsilon_j+\varepsilon_{j+1}$ have to be contained in the same summand, say in $V$. Hence the socles
	$$\text{soc}(V)\ =\ \text{soc}\left(\langle f_i^s, \varepsilon_i+\varepsilon_{i+1} \mid i\in\ZZ/m\ZZ, 1\le s\le N-1 \rangle\right)\ =\ \langle \varepsilon_j+\varepsilon_{j+1} \mid j\in\ZZ/m\ZZ\rangle$$
	agree, which implies that $\text{soc}(W)=0$, thus $W=0$. 
	
	Finally, we observe that none of the $\varepsilon_j$ is contained in  $\langle 1\rangle\oplus \langle f_i^s, \varepsilon_i+\varepsilon_{i+1} \mid i,s \rangle$. By adding the simple module $\langle \varepsilon_j\rangle$ for some $j$, say $j=0$, we obtain the decomposition above.
\end{proof}

\begin{thm}
For $m\ge 3$ odd, $\hh^0(\alg)$ has the following decomposition into indecomposable $\hh^1(\alg)$-modules:
\begin{align*}
\hh^0(\alg)\ =\ \langle 1\rangle\ \oplus\ \bigoplus\limits_{i\in\ZZ/m\ZZ} \langle f_i^s, \varepsilon_i+\varepsilon_{i+1} \mid 1\le s\le N-1\rangle.
\end{align*}
\end{thm}
\begin{proof}
Indeed, we see that
$$\varepsilon_0 = \frac12 \sum_{j\in\ZZ/m\ZZ} (-1)^{[j]}(\varepsilon_j+\varepsilon_{j+1})$$ is contained in 
the sum $\sum\limits_{i\in\ZZ/m\ZZ} \langle f_i^s, \varepsilon_i+\varepsilon_{i+1}\ |\ 1\le s\le N-1\rangle$. 
We conclude that this sum contains $\varepsilon_i$ for every $i$, 
hence $\hh^0(\alg)$ equals the sum of the submodules on the right hand side in the theorem. For dimension reasons, this sum is direct.
Finally, it is clear that the summands are indecomposable.
\end{proof}



\begin{thebibliography}{KKZ}


\bibitem{agaoka} Y.\ Agaoka, \textit{On the variety of 3-dimensional Lie algebras}, Lobachevskii J. Math.\ \textbf{3} (1999), 5--17.

\bibitem{BKL} D.\ Benson, R.\ Kessar and M.\ Linckelmann, \textit{On blocks of defect two and one simple module, and Lie algebra structure of $HH^1$}, 
J.\ Pure Appl.\ Algebra \textbf{221} (2017), 2953--2973.

\bibitem{Block} R.\ Block, \textit{The irreducible representations of the Lie algebra $\mathfrak{sl}(2)$ and of the Weyl algebra}, Adv.\
Math.\ \textbf{139} (1) (1981), 69--110.

\bibitem{BGSS} R.-O.\ Buchweitz, E.\ L.\ Green, N.\ Snashall, and \O.\ Solberg,
\textit{Multiplicative structures for Koszul algebras}, Q.\ J.\ Math.\ \textbf{59} (4)
(2008), 441--454. 

\bibitem{ChK} W.\ Chin and L.\  Krop, \textit{Representation theory of liftings of quantum planes}, 
J.\ Algebra Appl.\ \textbf{8} (2009), no. 2, 259--287. 

\bibitem{EGST} K.\ Erdmann, E. L. Green, N. Snashall and R.\ Taillefer, \textit{
Representation theory of the Drinfeld doubles of a family of Hopf algebras}, 
J.\ Pure Appl.\ Algebra \textbf{204} (2006), no. 2, 413--454.

\bibitem{ES} K.\ Erdmann and S.\ Schroll, \textit{On the Hochschild cohomology of tame Hecke algebras}, 
Arch.\ Math.\ (Basel) \textbf{204} (2010), no. 2, 117--127.

\bibitem{FS} R.\ Farnsteiner and A.\ Skowro\'nski, \textit{Classification of restricted Lie algebras with tame principal block}, 
J.\ Reine Angew.\ Math.\ 546 (2002), 1--45. 

\bibitem{FS2} R.\ Farnsteiner and A.\ Skowro\'nski, \textit{The tame infinitesimal groups of odd characteristic}, 
Adv.\ Math.\ \textbf{205} (2006), no. 1, 229--274. 

\bibitem{GTW} A.\ Gadbled, A.-L.\ Thiel, E.\ Wagner, \textit{Categorical action of the extended braid group of affine type $A$},
Commun.\ Contemp.\ Math.\  \textbf{19} (2017), 1650024, 39 pp.

\bibitem{G} M.\ Gerstenhaber, \textit{The cohomology structure of an associative ring}, Ann.\ Math.\  \textbf{78} (2) (1963), 267--288. 

\bibitem{G2} M.\ Gerstenhaber, \textit{On the deformation of rings and algebras}, Ann.\ Math.\ \textbf{79} (2) (1964), 59--103. 

\bibitem{Ho1} T.\ Holm, \textit{The even Hochschild cohomology ring of a block with cyclic defect group}, J.\ Algebra \textbf{178} (1995), 317--341.

\bibitem{Ho2} T.\ Holm, \textit{Hochschild cohomology of tame blocks}, J.\ Algebra \textbf{271} (2004), 798--826.

\bibitem{HK} R.\ Huerfano, M.\ Khovanov, \textit{A category for the adjoint representation}, J.\ Algebra \textbf{246} (2001), 514--542.

\bibitem{LLZ} G.\ Liu, R.\ Lu, K.\ Zhao, \textit{A class of simple weight Virasoro modules}, J.\ Algebra\ \textbf{424} (2015), 506--521. 

\bibitem{M} S.\ Mac Lane, ``Homology", Springer-Verlag, 1995. 

\bibitem{Mazorchuk} V.\ Mazorchuk, ``Lectures on $\mathfrak{sl}_2(\mathbb{C})$-modules", Imperial College Press, London, 2010.

\bibitem{Mazorchuk-Zhao} V.\ Mazorchuk, K.\ Zhao, \textit{Simple Virasoro modules which are locally finite over a positive part},
Selecta Math.\ (N.S.) \textbf{20} (3) (2014), 839--854.

\bibitem{P} M.\ K.\ Patra, \textit{On the structure of nonsemisimple Hopf algebras}, J.\ Phys.\ A \textbf{32} (1999), no. 1, 159--166. 

\bibitem{NW} C.\ Negron and S.\ Witherspoon, \textit{An alternate approach to the Lie bracket on Hochschild cohomology}, Homology Homotopy Appl.\ \textbf{18} (1) (2016), 265--285.

\bibitem{ST} N.\ Snashall and R.\ Taillefer, \textit{The Hochschild cohomology ring of a class of special biserial algebras}, J.\ Algebra\ Appl.\ \textbf{9} (1) (2010), 73--122.

\bibitem{S} M.\ Su\'{a}rez-\'{A}lvarez, \textit{A little bit of extra functoriality for Ext and the computation of the Gerstenhaber bracket}, J. Pure Appl. Algebra\ \textbf{221} (2017), 1981--1998.

\bibitem{Su} R.\ Suter, \textit{Modules over $\mathfrak{U}_q( \mathfrak{sl}_2)$}, 
Comm.\ Math.\ Phys.\ \textbf{163} (1994), no. 2, 359--393. 

\bibitem{V} Y.\ Volkov, \textit{Gerstenhaber bracket on the Hochschild cohomology via 
an arbitrary resolution}, \texttt{arXiv:1610.05741}.

\bibitem{X} J.\ Xiao, \textit{Finite-dimensional representations of $\mathfrak{U}_t(\mathfrak{sl}(2))$ at roots of unity}, 
Canad.\ J.\ Math. \textbf{49} (1997), no. 4, 772--787. 

\bibitem{X} Y.\ Xu, \textit{Hochschild cohomology of special biserial algebras}, 
Sci.\ China Ser.\ A \textbf{50} (2007), no. 8, 1117--1128. 


\end{thebibliography}
\end{document}